\documentclass[a4paper, reqno, 10pt]{amsart}
\usepackage{stmaryrd}
\usepackage{float}

\usepackage[usenames,dvipsnames]{color}
\usepackage{amsthm,amsfonts,amssymb,amsmath,amsxtra}
\usepackage{tikz}
\usepackage{tkz-euclide}
\usepackage[all]{xy}
\SelectTips{cm}{}
\usepackage{xr-hyper}
\usepackage[colorlinks=
citecolor=Black,
linkcolor=Red,
urlcolor=Blue]{hyperref}
\usepackage{verbatim}

\usepackage[margin=1.25in]{geometry}
\usepackage{mathrsfs}

\RequirePackage{xspace}
\RequirePackage{etoolbox}
\RequirePackage{varwidth}
\RequirePackage{enumitem}
\RequirePackage{tensor}
\RequirePackage{mathtools}
\RequirePackage{longtable}
\RequirePackage{multirow}

\def\ge{\geqslant}
\def\le{\leqslant}
\def\a{\alpha}

\def\g{\gamma}
\def\G{\Gamma}
\def\d{\delta}
\def\D{\Delta}

\def\e{\varepsilon}

\def\s{\sigma}
\def\t{\tau}

\def\k{\kappa}
\def\l{\lambda}

\def\i{^{-1}}

\def\<{\langle}
\def\>{\rangle}

\newcommand{\bJ}{\mathbf J}

\newcommand{\bA}{\mathbf A}

\newcommand{\bG}{\mathbf G}

\newcommand{\bM}{\mathbf M}

\newcommand{\bP}{\mathbf P}
\newcommand{\bN}{\mathbf N}
\newcommand{\bS}{\mathbf S}
\newcommand{\bT}{\mathbf T}

\newcommand{\Gr}{\mathrm{Gr}}
\newcommand{\Fl}{\mathrm{Fl}}

\newcommand{\rs}{\mathrm{rs}}

\newcommand{\BA}{\ensuremath{\mathbb {A}}\xspace}

\newcommand{\BC}{\ensuremath{\mathbb {C}}\xspace}

\newcommand{\BF}{\ensuremath{\mathbb {F}}\xspace}
\newcommand{{\BG}}{\ensuremath{\mathbb {G}}\xspace}

\newcommand{{\BK}}{\ensuremath{\mathbb {K}}\xspace}

\newcommand{\BN}{\ensuremath{\mathbb {N}}\xspace}

\newcommand{\BQ}{\ensuremath{\mathbb {Q}}\xspace}
\newcommand{\BR}{\ensuremath{\mathbb {R}}\xspace}
\newcommand{\BS}{\ensuremath{\mathbb {S}}\xspace}

\newcommand{\BZ}{\ensuremath{\mathbb {Z}}\xspace}
\def\bH{\mathbf H}

\newcommand{\CA}{\ensuremath{\mathcal {A}}\xspace}
\newcommand{\CB}{\ensuremath{\mathcal {B}}\xspace}

\newcommand{\CI}{\ensuremath{\mathcal {I}}\xspace}

\newcommand{\CK}{\ensuremath{\mathcal {K}}\xspace}
\newcommand{\CL}{\ensuremath{\mathcal {L}}\xspace}

\newcommand{\CO}{\ensuremath{\mathcal {O}}\xspace}
\newcommand{\CP}{\ensuremath{\mathcal {P}}\xspace}

\newcommand{\Ad}{{\mathrm{Ad}}}
\newcommand{\ad}{{\mathrm{ad}}}

\DeclareMathOperator{\Aut}{Aut}

\newcommand{\af}{\mathrm{af}}

\newcommand{\ex}{{\mathrm{ex}}}

\newcommand{\id}{\ensuremath{\mathrm{id}}\xspace}

\DeclareMathOperator{\rank}{rank}

\newcommand{\reg}{{\mathrm{reg}}}

\def\tW{\tilde W}
\def\tS{\tilde \BS}
\def\kk{\mathbf k}
\DeclareMathOperator{\supp}{supp}
\newtheorem{theorem}{Theorem}
\newtheorem{proposition}[theorem]{Proposition}
\newtheorem{lemma}[theorem]{Lemma}

\theoremstyle{definition}

\newtheorem{example}[theorem]{Example}

\newtheorem{remark}[theorem]{Remark}

\newtheoremstyle{query}%
{}{}%space above/below
{\color{red}}%body style
{}%heading indent
{\sffamily\bfseries}{:}{12pt}%heading style/punctuation/space after
{}% head spec
\theoremstyle{query}
\newtheorem{aq}{Author Query/Comment}

\newcommand{\baq}{\begin{aq}}%This just makes things easier
\newcommand{\eaq}{\end{aq}}

\numberwithin{equation}{section}
\numberwithin{theorem}{section}

\usepackage{lineno}
\setitemize[0]{leftmargin=*,itemsep=\the\smallskipamount}
\setenumerate[0]{leftmargin=*,itemsep=\the\smallskipamount}

\renewcommand{\to}{%
   \ifbool{@display}{\longrightarrow}{\rightarrow}%
   }
\let\shortmapsto\mapsto
\renewcommand{\mapsto}{%
   \ifbool{@display}{\longmapsto}{\shortmapsto}%
   }
\newlength{\olen}
\newlength{\ulen}
\newlength{\xlen}
\newcommand{\xra}[2][]{%
   \ifbool{@display}%
      {\settowidth{\olen}{$\overset{#2}{\longrightarrow}$}%
       \settowidth{\ulen}{$\underset{#1}{\longrightarrow}$}%
       \settowidth{\xlen}{$\xrightarrow[#1]{#2}$}%
       \ifdimgreater{\olen}{\xlen}%
          {\underset{#1}{\overset{#2}{\longrightarrow}}}%
          {\ifdimgreater{\ulen}{\xlen}%
             {\underset{#1}{\overset{#2}{\longrightarrow}}}
             {\xrightarrow[#1]{#2}}}}%
      {\xrightarrow[#1]{#2}}
   }
\makeatother
\newcommand{\xyra}[2][]{%
   \settowidth{\xlen}{$\xrightarrow[#1]{#2}$}%
   \ifbool{@display}%
      {\settowidth{\olen}{$\overset{#2}{\longrightarrow}$}%
       \settowidth{\ulen}{$\underset{#1}{\longrightarrow}$}%
       \ifdimgreater{\olen}{\xlen}%
          {\mathrel{\xymatrix@M=.12ex@C=3.2ex{\ar[r]^-{#2}_-{#1} &}}}%
          {\ifdimgreater{\ulen}{\xlen}%
             {\mathrel{\xymatrix@M=.12ex@C=3.2ex{\ar[r]^-{#2}_-{#1} &}}}
             {\mathrel{\xymatrix@M=.12ex@C=\the\xlen{\ar[r]^-{#2}_-{#1} &}}}}}%
      {\mathrel{\xymatrix@M=.12ex@C=\the\xlen{\ar[r]^-{#2}_-{#1} &}}}%
   }
\makeatletter
\newcommand{\xla}[2][]{%
   \ifbool{@display}%
      {\settowidth{\olen}{$\overset{#2}{\longleftarrow}$}%
       \settowidth{\ulen}{$\underset{#1}{\longleftarrow}$}%
       \settowidth{\xlen}{$\xleftarrow[#1]{#2}$}%
       \ifdimgreater{\olen}{\xlen}%
          {\underset{#1}{\overset{#2}{\longleftarrow}}}%
          {\ifdimgreater{\ulen}{\xlen}%
             {\underset{#1}{\overset{#2}{\longleftarrow}}}
             {\xleftarrow[#1]{#2}}}}%
      {\xleftarrow[#1]{#2}}
   }
\newcommand{\isoarrow}{%
   \ifbool{@display}{\overset{\sim}{\longrightarrow}}{\xrightarrow\sim}%
   }
   
\begin{document}

\title[]{On affine Lusztig varieties\\[.6cm]Sur les vari\'et\'es de Lusztig affines}
\author[Xuhua He]{Xuhua He}
\address{Department of Mathematics and New Cornerstone Science Laboratory, The University of Hong Kong, Pokfulam, Hong Kong, Hong Kong SAR, China}
\email{xuhuahe@hku.hk}

\thanks{}


\date{\today}

\begin{abstract}
Affine Lusztig varieties encode the orbital integrals of Iwahori--Hecke functions and serve as building blocks for the (conjectural) theory of affine character sheaves. We establish a close relationship between affine Lusztig varieties and affine Deligne--Lusztig varieties. Consequently, we give an explicit nonemptiness pattern and dimension formula for affine Lusztig varieties in most cases. 

\vskip 0.3cm

Les vari\'et\'es de Lusztig affines codent les int\'egrales orbitales des fonctions d'Iwahori--Hecke et servent d'\'el\'ements de base pour la théorie (conjecturale) des faisceaux caract\`eres affines. Nous établissons une relation \'etroite entre les variétés de Lusztig affines et les vari\'et\'es de Deligne--Lusztig. En cons\'equence, nous donnons un modèle explicite de non-vacuit\'e et une formule de dimension pour les vari\'et\'es de Lusztig affines dans la plupart des cas. 
\end{abstract}

\maketitle

%\tableofcontents

\section{Introduction}

\subsection{Deligne--Lusztig varieties and Lusztig varieties} Let $\bH$ be a connected reductive group over a finite field $\BF_q$. Let $H=\bH(\bar \BF_q)$ and $\operatorname{Fr}$ be the Frobenius endomorphism on $H$. Let $B$ be an $F$-stable Borel subgroup of $H$ and $W$ be the Weyl group of $H$. The Deligne--Lusztig variety is defined by $$X_w=\{h B; h \i \operatorname{Fr}(h) \in B \dot w B\},$$ where $w \in W$ and $\dot w$ is a representative of $w$ in $G$. It is the subvariety of the flag variety $H/B$ consisting of points with fixed relative position with its image under the Frobenius endomorphism.  Deligne--Lusztig varieties play a crucial role in the representation theory of finite groups of Lie type~\cite{DL}. 

Lusztig varieties were introduced in 1977 by Lusztig~\cite{Lu79} at the {\it LMS symposium on representations of Lie groups} in Oxford. They are defined by replacing the Frobenius endomorphism in the definition of $X_w$ with the conjugation action by a given regular semisimple element $h \in H$. It was shown in~\cite{Lu79} that the character values arising from the cohomology of  Deligne--Lusztig varieties are ``universal invariants'', which also make sense over complex numbers. The definition of Lusztig varieties was generalized to an arbitrary element $h \in H$ in Lusztig's theory of character sheaves~\cite{Lu-char}. Character sheaves can be defined in two different ways: via  ``admissible complexes'' and via  Lusztig varieties. It is a deep result that the two definitions coincide. Abreu and Nigro ~\cite{AN} used Lusztig varieties in their geometric approach to characters of Hecke algebras. The homology classes of Deligne--Lusztig varieties and those of Lusztig varieties were studied by Kim ~\cite{Kim}.

\subsection{Affine Lusztig varieties} The main subjects of this paper are the affine analogs of Lusztig varieties. 

Let $\kk$ be an algebraically closed field and $L=\kk(\!(\e)\!)$ be the field of the Laurent series (equal characteristic case) or $W(\kk)[1/p]$ (mixed characteristic case) if $\kk$ is characteristic $p$, where $W(\kk)$ is the ring of $p$-typical Witt vectors. Let $\bG$ be a connected reductive group over $L$. Let $\breve G=\bG(L)$. Let $\Gr=\breve G/\CK^{\mathrm{sp}}$ be the affine Grassmannian of $\breve G$ and $\Fl=\breve G/\CI$ be the affine flag variety. We define the affine Lusztig varieties 
\begin{gather*} 
Y^{\bG}_\mu(\g)=\{g \CK^{\mathrm{sp}}; g \i \g g \in \CK^{\mathrm{sp}} t^\mu \CK^{\mathrm{sp}}\}, \\
Y_w^{\bG}(\g)=\{g \CI \in \Fl; g \i \g g \in \CI \dot w \CI\}
\end{gather*}
in the affine Grassmannian and the affine flag variety, respectively. Here $\g$ is a regular semisimple element in $\breve G$, $\mu$ is a dominant coweight of $\breve G$, and $w$ is an element in the Iwahori--Weyl group of $\breve G$. Lusztig first introduced affine Lusztig varieties in~\cite{Lu11}. Affine Lusztig varieties in the affine Grassmannian were also studied by Kottwitz and Viehmann ~\cite{KV}. In the literature,  affine Lusztig varieties are also called  generalized affine Springer fibers or  Kottwitz--Viehmann varieties. 

In equal characteristic setting,  affine Lusztig varieties are locally closed subschemes of the affine Grassmannian and affine flag variety, equipped with a reduced scheme structure. In mixed characteristic setting, we consider affine Lusztig varieties as perfect schemes in the sense of Zhu~\cite{Zhu} and Bhatt--Scholze~\cite{BS}, i.e., as locally closed perfect subschemes of  $p$-adic flag varieties. 

Affine Lusztig varieties in the affine Grassmannian and the affine flag variety arise naturally in the study of orbital integrals of spherical Hecke algebras and Iwahori--Hecke algebras. Roughly speaking, based on the Grothendieck-Lefschetz trace formula, one may interpret the orbital integrals of spherical Hecke algebras (resp. Iwahori--Hecke algebras) as point-counting problems on the associated affine Lusztig varieties in the affine Grassmannian (resp. the affine flag variety). We refer to Yun's lecture notes~\cite{Yun}, and recent preprint of Chi~\cite[\S 5.3]{Chi24} for some formulas cononecting the affine Springer fibers and orbital integrals and the thesis work of G. Wang~\cite[Proposition 11.2.5]{Wang} for some formulas connecting the affine Lusztig varieties in the affine Grassmannian and the orbital integrals. 

Affine Lusztig varieties also serve as building blocks for the (conjectural) theory of affine character sheaves. Namely, we consider the map
\[
m: \breve G \times^{\CI} \CI \dot w \CI \to \breve G, \qquad (g, g') \mapsto g g' g \i,
\] where $\breve G \times^{\CI} \CI \dot w \CI$ is the quotient of $\breve G \times \CI \dot w \CI$ by the $\CI$-action defined by $i \cdot (g, g')=(g i \i, i g' i \i)$. The fibers of this map are the affine Lusztig varieties in the affine flag variety. This map is the affine analog of Lusztig's map to define the character sheaves for reductive groups over algebraically closed fields. One may define a similar map by replacing $\CI$ with $\CK^{\mathrm{sp}}$. It is interesting to realize Lusztig's unipotent almost characters of $p$-adic groups~\cite{Lu12} using  affine Lusztig varieties. 

\subsection{Main result} It is a fundamental question to determine the nonemptiness pattern and dimension formula for affine Lusztig varieties. 

This question was solved for affine Lusztig varieties in the affine Grassmannian of a split connected reductive group in equal characteristic case (under a mild assumption on the residue characteristic) by Bouthier and Chi ~\cite{Bou15,BC,Chi}. Their method relies on a global argument using the Hitchin fibration and does not work in mixed characteristic case (e.g., for $Y_\mu(\g)$ for a ramified element $\g$). See~\cite[Remark 1.2.3]{Chi}. Little is known for nonsplit groups or for the affine flag case. 

The main purpose of this paper is to study  affine Lusztig varieties in the affine flag variety of any (not necessarily split) connected reductive group in equal and mixed characteristic cases. Our approach is very different from  previous approaches. The key idea is to establish a close relationship between  affine Lusztig varieties and their Frobenius-twisted analogs,  affine Deligne--Lusztig varieties. Our approach is purely local and thus works also in mixed characteristic case. We will focus on the affine flag case, since the affine Grassmannian case can be deduced from this. 

We first give a quick review of affine Deligne--Lusztig varieties. Let $F'$ be a  non-archimedean local field and $\breve F'$ be the completion of the maximal unramified extension of $F'$. Let $\bG'$ be a connected reductive group over $F'$, $\breve G'=\bG'(\breve F')$ and $\s$ be the Frobenius endomorphism on $\breve G'$. Let $\CI'$ be a $\s$-stable Iwahori subgroup of $\breve G'$. Let $b \in \breve G'$ and $w$ be an element in the Iwahori--Weyl group of $\breve G'$. The affine Deligne--Lusztig variety $X_w(b)$ is defined by $$X^{\bG'}_w(b)=\{g \CI' \in \Fl'; g \i b \s(g) \in \CI' \dot w \CI'\}.$$ The notion of an affine Deligne--Lusztig variety was first introduced by Rapoport ~\cite{Ra}. It plays an important role in arithmetic geometry and the Langlands program. Affine Deligne--Lusztig varieties have been studied extensively in the past two decades, and we now have a much better understanding of these varieties than of affine Lusztig varieties. The nonemptiness pattern and the dimension formula for  affine Deligne--Lusztig varieties are completely known in the affine Grassmannian case and are known for most pairs $(w, b)$ in the affine flag case. We refer to the survey article~\cite{He-ICM} for a detailed discussion. 

The structure of  affine Deligne--Lusztig varieties, in general, is quite complicated, and affine Lusztig varieties are even more complicated, partially because,  in the loop group, the classification of  ordinary conjugacy classes is more complicated than the classification of  Frobenius-twisted conjugacy classes. In the special case where $w=1$, the associated affine Deligne--Lusztig variety is either empty or a discrete set. The associated affine Lusztig variety is the affine Springer fiber $\Fl^\g$ and is, in general, finite-dimensional (but not zero-dimensional). The dimension formula for affine Springer fiber was obtained by Bezrukavnikov~\cite{Be} in equal characteristic case and Chi~\cite{Chi24} in mixed characteristic case.

The main result of this paper, roughly speaking, shows that the difference between affine Lusztig varieties and  affine Deligne--Lusztig varieties comes from the affine Springer fibers. 

\begin{theorem}[Theorem~\ref{thm:le}]
Suppose that the group $\bG'$ over $F'$ is associated with the group $\bG$ over $L$. Let $\g$ be a regular semisimple element of $\breve G$ and $[b]=f_{\bG, \bG'}(\{\g\})$. Then, for any $w \in \tW$, we have 
\begin{enumerate}
    \item $Y^{\bG}_w(\g) \neq \emptyset$ if and only if $X^{\bG'}_w(b) \neq \emptyset$;

    \item if $X^{\bG'}_w(b) \neq \emptyset$, then $\dim Y^{\bG}_w(\g)=\dim X^{\bG'}_w(b)+\dim Y^{\bG}_{\g}$.
\end{enumerate}

In particular, $Y^{\bG}_w(\g)$ is either an empty set or finite-dimensional. 
\end{theorem}

%We now explain this statement in detail. 

%By convention, we set $\dim \emptyset=-\infty$ and $-\infty+n=-\infty$ for any $n \in \BN$. In particular, the affine Lusztig variety on the left-hand side is empty if and only if the affine Deligne--Lusztig variety on the right-hand side is empty. This theorem also provides information on the nonemptiness pattern. 

The notion of associated groups is given in \S\ref{sec:ass}. It means that $\bG'$ is residually split and there is a natural bijection between the Iwahori--Weyl groups of $\bG$ and $\bG'$. If the residue characteristic of $\breve F$ is positive, then we may take $\bG'$ over $\breve F$ so that $\breve G=\breve G'$. However, if $\bG$ is defined over $\BC(\!(\e)\!)$, then  $\bG'$ has to be chosen over a different field. The isomorphism class of the affine Lusztig variety $Y_w(\g)$ depends on $w$ and the conjugacy class $\{\g\}$ of $\g$ in $\breve G$, and the isomorphism class of the affine Deligne--Lusztig variety $X_w(b)$ depends on $w$ and the Frobenius-twisted conjugacy class $[b]$ of $b$ in $\breve G'$. To match the data, we introduce in \S\ref{sec:ass} a surjective (but not injective) map $f_{\bG, \bG'}$ from the set of conjugacy classes of $\breve G$ to the set of Frobenius-twisted conjugacy classes of $\breve G'$. The variety $Y_\g$ is the affine Springer fiber associated with $\g$, introduced in \S\ref{sec:affineS} via the Hodge--Newton decomposition for ordinary conjugation action established in Theorem~\ref{thm:HN1}. In the case where $\g$ is bounded, $Y_\g$ is the affine Springer fiber of $\bG$, and in general, $Y_\g$ is the affine Springer fiber of some Levi subgroup of $\bG$. 

Finally, in Theorems~\ref{thm:gasf-gr} and~\ref{thm:gasf-fl}, we provide explicit descriptions of the nonemptiness pattern and the dimension formula of  affine Lusztig varieties, in most cases based on our knowledge of affine Deligne--Lusztig varieties and affine Springer fibers. 

One may also apply the method presented in this paper to study the irreducible components. This will be done in future work. 

It is also worth pointing out that the relation between affine Lusztig varieties and affine Deligne--Lusztig varieties is {\it not} constructed in a geometric way, that is, by explicitly defining a map from one type of variety to another and then computing images and fibers. One may expect there to be some close relation between the (co)homology of affine Lusztig varieties and that of  affine Deligne--Lusztig varieties, similar to the work of Lusztig~\cite{Lu79} (see, e.g., the equality displayed in~\cite[page 13]{Lu17}). Another possible way to reveal the relation between affine Lusztig varieties and affine Deligne--Lusztig varieties is via the (conjectural) connection between the cocenter of the affine Hecke category and the (conjectural) theory of affine character sheaves. The cocenter of the affine Hecke category has recently been studied by Li, Nadler, and Yun ~\cite{LNY}. 

\smallskip

{\bf Acknowledgement: } XH is partially supported by the New Cornerstone Science Foundation through the New Cornerstone Investigator Program and the Xplorer Prize, and by Hong Kong RGC grant 14300122. We thank Jingren Chi, George Lusztig, Michael Rapoport and Zhiwei Yun for helpful discussions and suggestions, and Felix Schremmer for careful reading of a preliminary version of this paper. We thank the editors and referees for many helpful suggestions.

\section{Preliminary}
\subsection{Loop groups and  Iwahori--Weyl groups} Let $\kk$ be an algebraically closed field and $L=\kk(\!(\e)\!)$ be the field of the Laurent series or $W(\kk)[1/p]$ if $\kk$ is characteristic $p$, where $W(\kk)$ is the ring of $p$-typical Witt vectors. Let $\bG$ be a connected reductive group over $L$. Let $\breve G=\bG(L)$ and let $\breve G^{\rs}$ be the set of regular semisimple elements in $\breve G$. 

Let $\bS$ be a maximal split torus of $\bG$ and $\bT$ be the centralizer of $\bS$ in $\bG$. Then $\bT$ is a maximal torus of $\bG$. The relative Weyl group of $\breve G$ is defined to be $W_0=\bN(\bT)(L)/\bT(L)$, where $\bN(\bT)$ is the normalizer of $\bT$. The Iwahori--Weyl group of $\breve G$ is defined as $\tW=\bN(\bT)(L)/\bT(L)_0$, where $\bT(L)_0$ is the (unique) parahoric subgroup of $\bT(L)$. Let $\CA$ be the apartment of $\bG$ corresponding to $\bS$. We fix an alcove $\mathfrak a$ in $\CA$ and let $\CI$ be the corresponding Iwahori subgroup of $\breve G$. The simple reflections are reflections along the wall of the base alcove $\mathfrak a$. Let $\tilde \BS$ be the set of simple reflections of $\tW$. We choose a special vertex $v_0$ of $\mathfrak a$ and represent $\tW$ as a semidirect product $$\tW=X_*(\bT)_{\G_0} \rtimes W_0=\{t^\l w; \l \in X_*(\bT)_{\G_0}, w \in W_0\},$$ where $X_*(\bT)_{\G_0}$ is the set of coinvariants of the coweight lattice $X_*(\bT)$ under the action of $\G_0=\text{Gal}(\bar L/L)$. Let $\BS=\tilde \BS \cap W_0$ be the set of simple reflections of $W_0$. For any $w \in \tW$, we choose a representative $\dot w$ in $\bN(\bT)(L)$. By convention, the dominant Weyl chamber of $V=X_*(T)_{\G_0} \otimes \BR$ is opposite to the unique Weyl chamber containing $\mathfrak a$. Let $\D$ be the set of relative simple roots determined by the dominant Weyl chamber. We denote by $w_0$ the longest element of $W_0$. 

A subset $K$ of $\tilde \BS$ is called {\it spherical} if the subgroup $W_K$ of $\tW$ generated by $s \in K$ is finite. In this case, we denote by $\CP_K$ the standard parahoric subgroup of $\breve G$ generated by $\CI$ and $\dot w$ for $w \in W_K$. By definition, a parahoric subgroup of $\breve G$ is a finite union of double cosets of an Iwahori subgroup. Any parahoric subgroup of $\breve G$ is conjugate to a standard parahoric subgroup $\CP_K$ for some spherical subset $K$ of $\tilde \BS$. 

For any $J \subset \tilde \BS$, we denote by ${}^J \tW$ (resp. $\tW^J$) the set of minimal representatives of $W_J \backslash \tW$ (resp. $\tW/W_J$). For any $J, K \subset \tilde \BS$, we simply write ${}^J \tW^K$ for ${}^J \tW \cap \tW^K$. 

Let $\breve G_{\af}$ be the subgroup of $\breve G$ generated by all the parahoric subgroups and $W_{\af}$ be the Iwahori--Weyl group of $\breve G_{\af}$. Then $W_{\af}$ is the subgroup of $\tW$ generated by $\tilde S$. Let $\Omega$ be the set of length-zero elements in $\tW$. Then we have $\tW=W_{\af} \rtimes \Omega$. 

\subsection{Affine Lusztig varieties} The affine flag variety of $\breve G$ is defined as $\Fl=\breve G/\CI$. If $\CK^{\mathrm{sp}} \supset \CI$ is the maximal special parahoric subgroup of $\breve G$ corresponding to the special vertex $v_0$, then we also consider the affine Grassmannian $\Gr=\breve G/\CK^{\mathrm{sp}}$. When $L$ is the field of Laurent series, both the affine flag variety and the affine Grassmannian have natural ind-scheme structures. In a mixed characteristic setting, we regard the affine flag variety and the affine Grassmannian as the perfect ind-schemes in the sense of Zhu~\cite{Zhu} and Bhatt–Scholze~\cite{BS}. 

Let $X_*(T)_{\G_0}^+$ be the set of dominant coweights. We have the decompositions $$\breve G=\bigsqcup_{w \in \tW} \CI \dot w \CI, \qquad \breve G=\bigsqcup_{\mu \in X_*(T)_{\G_0}^+} \CK^{\mathrm{sp}} t^\mu \CK^{\mathrm{sp}}.$$

For any $\g \in \breve G$ and $w \in \tW$, the {\it affine Lusztig variety} $Y_w(\g)$ in the affine flag variety $\Fl$ is defined by $$Y_w(\g)=\{g \CI \in \Fl; g \i \g g \in \CI \dot w \CI\}.$$ For $w \in \tW$ with $\ell(w)=0$ and $\g \in \CI \dot w \CI$, we have $\CI \dot w \CI=\CI \g$, and $Y_w(\g)=\Fl^{\g}$ is the affine Springer fiber. 

For any dominant coweight $\mu$ and $\g \in \breve G$, the {\it affine Lusztig variety} $Y_\mu(\g)$ in the affine Grassmannian $\Gr$ is defined by $$Y_\mu(\g)=\{g \CK^{\mathrm{sp}} \in \Gr; g \i \g g \in \CK^{\mathrm{sp}} t^\mu \CK^{\mathrm{sp}}\}.$$ 

By definition, $Y_w(\g)$ is a locally closed ind-subscheme of $\Fl$, and $Y_\mu(\g)$ is a locally closed ind-subscheme of $\Gr$. Let $Z_{\breve G}(\g)=\{g \in \breve G; g \i \g g=\g\}$ be the centralizer of $\g$. It acts on the affine Lusztig varieties $Y_w(\g)$ and $Y_\mu(\g)$.  

\subsection{Affine Deligne--Lusztig varieties}

Let $F'=\mathbb F_q(\!(\e)\!)$ and $\breve F'=\overline{\mathbb F}_q(\!(\e)\!)$. Let $\bG'$ be a connected reductive group, residually split over $F'$ and $\breve G'=\bG'(\breve F')$. Let $\s$ be the Frobenius endomorphism on $\breve G'$. We fix a $\s$-stable Iwahori--Weyl group $\CI'$ of $\breve G'$. Let $\tW'$ be the Iwahori--Weyl group of $\breve G'$ and $W'_0$ be the relative Weyl group. 

For $b \in \breve G'$ and $w \in \tW'$, the affine Deligne--Lusztig variety $X_w(b)$ in the affine flag variety is defined by $$X_w(b)=\{g \CI\in \Fl'; g \i b \s(g) \in \CI' \dot w \CI'\}.$$

Let $\CK' \supset \CI'$ be a $\s$-stable maximal special parahoric subgroup of $\breve G'$; then we may also consider the affine Deligne--Lusztig variety in the affine Grassmannian $\Gr'=\breve G'/\CK'$ defined by $$X_\mu(b)=\{g \CK' \in \Gr'; g \i b \s(g) \in \CK' t^\mu \CK'\}.$$

It is known that affine Deligne--Lusztig varieties are subschemes of locally finite type in the affine flag variety (resp. affine Grassmannian). Let $\bJ_b$ be the $\s$-centralizer of $b$ and $J_b=\{g \in \breve G'; g \i b \s(g)=b\}$ be the group of $F'$-rational points of $\bJ_b$. Then $J_b$ acts naturally on the affine Deligne--Lusztig varieties associated with $b$. 

\subsection{The sets $B_\s(\bG')$ and $B(\bG)$}
For any $\g \in \breve G$, let $\{\g\}=\{g \g g \i; g \in \breve G\}$ be the conjugacy class of $\g$. Let $B(\bG)$ be the set of conjugacy classes of $\breve G$. For any $b \in \breve G'$, let $[b]=\{g b \s(g) \i; g \in \breve G'\}$ be the $\s$-conjugacy class of $b$. Let $B_\s(\bG')$ be the set of $\s$-conjugacy classes of $\breve G'$.  
 
Note that the isomorphism classes of affine Deligne--Lusztig varieties depend on the $\s$-conjugacy class $[b]$, not the representative $b$ in $[b]$. Similarly, the isomorphism classes of affine Lusztig varieties depend on the conjugacy class $\{\g\}$, not the representative $\g$ of $\{\g\}$. We now recall Kottwitz's classification of $B_\s(\bG')$ and its analog on $B(\bG)$. 

Fix a maximal $F'$-split torus $\bA'$ of $\bG'$. Let $\bS'$ be a maximal $\breve F'$-split torus that is defined over $F'$ and contains $\bA'$. Let $\G'_0$ be the absolute Galois group of $F'$. Let $\bT'$ be the centralizer of $\bS'$. Let $X_*(\bT')_\BQ^+$ be the set of dominant rational coweights on $\bT'$. 
The $\s$-twisted conjugacy classes of $\breve G'$ have been classified by Kottwitz~\cite{kottwitz-isoI,kottwitz-isoII}. We only state the result for residually split groups here. 

\begin{theorem}
There is an embedding $$f_\s: B_\s(\bG') \to \pi_1(\bG')_{\G_0'} \times X_*(\bT')_\BQ^+, \quad [b] \mapsto (\k([b]), \nu_{[b]}).$$ 
\end{theorem}

Chai ~\cite{Chai} gave an explicit description of the image of $f_\s$.

Similarly, we have the following.

\begin{theorem}
There is a natural map $$f: B(\bG) \to \pi_1(\bG)_{\G_0} \times X_*(\bT)_\BQ^+, \quad \{\g\} \mapsto (\k(\{\g\}), \nu_{\{\g\}}).$$
\end{theorem}

This is due to Kottwitz and Viehmann~\cite{KV} for split groups. The general case was studied in~\cite[\S 3]{HN20}. Note that the map $f$, unlike $f_\s$, is not injective in general. We will not recall the precise definition of the maps $f$ and $f_\s$, but we will give explicit descriptions of the restriction of these maps to the Iwahori--Weyl groups in the next subsection. 

\subsection{The straight conjugacy classes}\label{two-diagram} To match the conjugacy classes of $\breve G$ with the $\s$-conjugacy classes of $\breve G'$, it is necessary to compare the images of $f$ with $f_\s$. In this subsection, we recall the definition of the straight conjugacy classes in the Iwahori--Weyl groups, and compare the set of straight conjugacy classes with the images of $f$ and $f_\s$. 

By definition, an element $w \in \tW$ is called {\it straight} if $\ell(w^n)=n \ell(w)$ for all $n \in \BN$. A conjugacy class of $\tW$ is called {\it straight} if it contains a straight element. We denote by $\tW \sslash \tW$ the set of straight conjugacy classes of $\tW$. 

Let $w \in \tW$. Then there exists a positive integer $n$ such that $w^n=t^\l$ for some $\l \in X_*(\bT)_{\G_0}$. We set $\nu_w=\frac{\l}{n} \in X_*(\bT)_{\BQ}$. It is easy to see that $\nu_w$ is independent of the choice of $n$. Let $\bar \nu_w$ be the unique dominant coweight in the $W_0$-orbit of $\nu_w$. We also have a natural identification $\tW/W_{\af} \cong \Omega \cong \pi_1(\bG)_{\G_0}$. We consider the maps $\tW \to \pi_1(\bG)_{\G_0} \times X_*(\bT)_\BQ^+$ given by $w \mapsto (w W_{\af}, \bar \nu_w)$. By~\cite[Theorem 3.3]{HN14}, this map induces an injection $\tW \sslash \tW \to \pi_1(\bG)_{\G_0} \times X_*(\bT)_\BQ^+$. Similarly, we have an injection $\tW' \sslash \tW' \to \pi_1(\bG')_{\G'_0} \times X_*(\bT')_\BQ^+$.

It is worth pointing out that the induced action of $\s$ on $\tW'$ is trivial, and thus $\tW' \sslash \tW'$ considered here is the same as the set of straight $\s$-conjugacy classes $\tW' \sslash \tW'$ in \cite{He14}. By~\cite[\S 3]{He14}, the map $w \mapsto \dot w$ induces a natural bijection between $B_\s(\bG')$ and $\tW' \sslash \tW'$ such that the following diagram is commutative:

\[
\xymatrix{
B_\s(\bG') \ar[rr] \ar@{^{(}->}[dr]_{f_\s} & & \tW' \sslash \tW' \ar@{^{(}->}[dl] \ar[ll] \\
& \pi_1(\bG')_{\G'_0} \times X_*(\bT')_\BQ^+
}
\]
The bijection is given as follows. Let $C'$ be a straight conjugacy class of $\tW'$. Then, by~\cite{He14}, $\breve G' \cdot_\s \CI' \dot w' \CI'$ is independent of the choice of the minimal length elements $w'$ in $C'$ and is a single $\s$-conjugacy class of $\breve G'$. The map $B_\s(\bG') \to \tW' \sslash \tW'$ is defined by $[b] \mapsto C_{[b]}$, where $C_{[b]}$ is the unique element in $\tW' \sslash \tW'$ with $[b]= \breve G' \cdot_\s \CI' \dot w' \CI'$ for any minimal length element $w'$ of $C_{[b]}$.

By~\cite[Section 3]{HN20}, we have a similar commutative diagram
\[
\xymatrix{
B(\bG) \ar@{->>}[rr] \ar[dr]_{f} & & \tW \sslash \tW \ar@{^{(}->}[dl] \\
& \pi_1(\bG)_{\G_0} \times X_*(\bT)_\BQ^+
}
\]
Here the map $B(\bG) \to \tW \sslash \tW$ is defined as follows. By~\cite[\S 3.2]{HN20}, given a straight conjugacy class $C$ of $\tW$, the subset $\breve G \cdot \CI \dot w \CI$ is independent of the choice of the minimal length elements $w$ in $C$. We denote this subset by $\{C\}$. Then, by~\cite[Theorem 3.2]{HN20}, $\breve G=\bigsqcup_{C \in \tW \sslash \tW} \{C\}$. The map $B(\bG) \to \tW \sslash \tW$ is defined by $\{\g\} \mapsto C_{\{\g\}}$, where $C_{\{\g\}}$ is the unique element in $\tW \sslash \tW$ with $\{\g\} \subset \{C_{\{\g\}}\}$.

The set $\{C\}$ is stable under the conjugation action of $\breve G$, but it may contain infinitely many conjugacy classes. The map $B(\bG) \to \tW \sslash \tW$ is surjective, but not injective. This is different from the bijective map $B_\s(\bG') \to \tW' \sslash \tW'$. 

\subsection{The map from $B(\bG)$ to $B_\s(\bG')$}\label{sec:ass}

We say that the group $\bG'$ over $F'$ is {\it associated with} the group $\bG$ over $L$ if $\bG'$ is residually split over $F'$ and we have a length-preserving isomorphism $\tW \cong \tW'$ that is compatible with the semidirect products, i.e.,
\[
\xymatrix{X_*(T)_{\G_0} \rtimes W_0 \ar[d]^{\cong} \ar@{=}[r] & \tW \ar[d]^{\cong} \\ X_*(T')_{\G'_0} \rtimes W'_0  \ar@{=}[r] & \tW'
}
\]
In particular, we require that this isomorphism induces isomorphisms on the coweight lattices and on the relative Weyl groups and a bijection on the set of simple reflections. By the classification of reductive groups, for any $\bG$, there exists a connected reductive group $\bG'$ over $F'$ that is associated with it. 

Suppose that $\bG'$ is associated with $\bG$. Then we may identify $\tW'$ with $\tW$. Under this identification, we have $\tW \sslash \tW=\tW' \sslash \tW'$. Thus there exists a map  $f_{\bG, \bG'}: B(\bG) \to B_\s(\bG')$ such that the following diagram commutes:
\[
\xymatrix{
B(\bG) \ar[rrr]^{f_{\bG, \bG'}} \ar[rdd] \ar[rd] & & & B_\s(\bG') \ar@{->>}[ldd] \ar[ld] \\
& \tW \sslash \tW \ar@{=}[r] \ar@{^{(}->}[d] & \tW' \sslash \tW' \ar[ru] \ar@{^{(}->}[d] & \\
& \pi_1(\bG)_{\G_0} \times X_*(\bT)_\BQ^+ & \pi_1(\bG')_{\G'_0} \times X_*(\bT')_\BQ^+ &
}
\]
Note that $B_\s(\bG')$ is in natural bijection with $\tW \sslash \tW$. Thus the map $f_{\bG, \bG'}: B(\bG) \to B_\s(\bG')$ is surjective. However, this map is not injective in general. 

\section{Affine Springer fibers}

\subsection{Bounded-modulo-center elements}
In the rest of this paper, unless otherwise stated, we assume that either $\mathrm{char}(\kk)=0$ or $\mathrm{char}(\kk)>0$ does not divide the Weyl group $W_0$ of $\bG$. 

In this section, we consider the group $\breve G$ together with a group automorphism $\d$ such that $\d(\CI)=\CI$. Then $\d$ also induces group automorphisms on $\tW$ and on $\Omega$. We consider the group extensions $$\breve G^{\ex}=\breve G \rtimes \<\d\>, \tW^{\ex}=\tW \rtimes \<\d\> \text{ and } \Omega^{\ex}=\Omega \rtimes \<\d\>.$$ 

An element of $\breve G^{\ex}$ is called {\it bounded} if it is contained in a bounded subgroup. Note that if $\t \in \Omega^{\ex}$ is of finite order, then any element in $\CI \dot \t$ is bounded. On the other hand, if $g \in \breve G^{\ex}$ is bounded, then it has a fixed point in the affine flag variety. So it is conjugate to an element of the form $\CI \dot \t$ for some $\t \in \Omega^{\ex}$. By the boundedness, $\t$ is of finite order. Thus an element of $\breve G^{\ex}$ is bounded if and only if it is conjugate to an element of $\CI \dot \t$ for some $\t \in \Omega^{\ex}$ with finite order. 

Let $\breve G_{\ad}$ be the adjoint group of $\breve G$. We have a natural projection map $\breve G^{\ex} \to \Aut(\breve G_{\ad})$. An element of $\breve G^{\ex}$ is called {\it bounded-modulo-center} if its image in $\Aut(\breve G_{\ad})$ is bounded. We assume furthermore that the image of $\d$ in $\Aut(\breve G_{\ad})$ is of finite order. Then an element of $\breve G^{\ex}$ is bounded-modulo-center if and only if it is conjugate to an element of $\CI \dot \t$ for some $\t \in \Omega^{\ex}$. As a consequence, an element of $\breve G^{\ex}$ is bounded-modulo-center if and only if it is conjugate to an element in $\CP_K \dot \t$ for some $\t \in \Omega^{\ex}$ and some $\Ad(\t)$-stable spherical subset $K$ of $\tilde \BS$. 

\subsection{Regular locus}\label{sec:reg}
Let $\g$ be a bounded-modulo center element of $\breve G^{\ex}$. Let $\t \in \Omega^{\ex}$ and $K$ be an $\Ad(\t)$-stable spherical subset of $\tilde \BS$. We define the {\it affine Springer fiber} in the partial affine flag variety $\breve G/\CP_K$ by $$Y_{K, \t}(\g)=\{g \CP_K \in \breve G/\CP_K; g \i \g g \in \CP_K \dot \t\}.$$ By definition, $Y_{K, \t}(\g) \neq \emptyset$ if and only if $\g$ is conjugate to an element in $\CP_K \dot \t$. If we assume furthermore that $\g \in \CP_K \dot \t$, then the condition that $g \i \g g \in \CP_K \dot \t$ is equivalent to the condition that $\g (g \CP_K) \g \i=g \CP_K$. In other words, $Y_{K, \t}(\g)=(\breve G/\CP_K)^\g$ is the subscheme of fixed points of $\g$ on $\breve G/\CP_K$. 

Let $\bar \CP_K$ be the reductive quotient of $\CP_K$, i.e., the quotient of $\CP_K$ by its pro-unipotent radical. The Dynkin diagram of $\bar \CP_K$ is $K$. The conjugation action of $\dot \t$ on $\CP_K$ induces a diagram automorphism on $\bar \CP_K$, which we denote by $\bar \t$. For any element $\g'$ in $\CP_K \dot \t$, we denote by $\bar \g'$ its image in the (not necessarily connected) reductive group $\bar \CP_K \rtimes \<\bar \t\>$. We consider the {\it regular locus} $$Y_{K, \t}^{\reg}(\g)=\{g \CP_K \in Y_{K, \t}(\g); \overline{g \i \g g} \text{ is regular in } \bar \CP_K \rtimes \<\bar \t\>\}.$$

Since the set of regular elements in any (not necessarily connected) reductive group over $\kk$ is open, $Y^{\reg}_{K, \t}(\g)$ is open in $Y_{K, \t}(\g)$. By~\cite[\S 3, Proposition 1]{KL}, the affine Springer fibers associated with the bounded-modulo-center regular semisimple elements are finite-dimensional. 

We now state the main result of this section. 

\begin{theorem}\label{thm: aff}
Let $\g \in \breve G^{\ex}$ be a bounded-modulo-center regular semisimple element. Then 
\begin{enumerate}
\item the affine Springer fiber $\Fl^\g$ in the affine flag variety is equidimensional;
\item  for any $\t \in \Omega^{\ex}$ and an $\Ad(\t)$-stable spherical subset $K$ of $\tilde \BS$ with $\g \in \CP_K \dot \t$, the regular locus $Y^{\reg}_{K, \t}(\g)$ is nonempty and $$\dim Y_{K, \t}(\g)=\dim Y^{\reg}_{K, \t}(\g)=\dim \Fl^\g.$$
\end{enumerate}
\end{theorem}

\subsection{Topological Jordan decomposition}
We follow~\cite[Appendix B]{BV}. Recall that $\kk$ is the residue field of $\breve F$. Let $\g \in \breve G^{\ex}$ be a bounded semisimple element. We say that $\g$ is {\it strongly semisimple} if for every representation of $\bG^{\ex}$ (over an algebraic closure of $\breve F$), any eigenvalue of the image of $\g$ is in $\kk^\times$ (in equal characteristic case), or the Teichm\"uller representative of an element in $\kk^\times$ (in mixed characteristic case). The latter condition is equivalent to the condition that any eigenvalue is of finite order prime to $p$, where $p$ is the residue characteristic. We say that $\g$ is {\it topologically unipotent} if for every representation of $\bG^{\ex}$ (over an algebraic closure of $\breve F$), any eigenvalue $\a$ of $\g$ satisfies $\mathrm{val}_{\breve F}(\a-1)>0$. 

The following result, which is called the {\it topological Jordan decomposition}, was proved by Bezrukavnikov and Varshavsky~\cite[Lemma B.1.6]{BV} for split groups in equal characteristic case. The same proof works for all groups in both equal and mixed characteristic cases.

\begin{lemma}\label{lem:jordan}
    Let $\g \in \breve G^{\ex}$ be a bounded semisimple element. Then there exists a unique decomposition $\g=s u=u s$ such that $s \in \breve G^{\ex}$ is strongly semisimple and $u \in \breve G^{\ex}$ is topologically unipotent. 
\end{lemma}

The topological Jordan decomposition can be seen explicitly as follows. Note that any bounded element in $\breve G^{\ex}$ is conjugate to an element in $\CI \dot \t$ for some $\t \in \Omega^{\ex}$ of finite order. Suppose that $\g \in \CI \dot \t$. Following the proof of~\cite[Lemma B.1.7]{BV}, we may represent $\CI$ as $\lim_n H_n$, where $H_n=\CI/\CI_{n}$ is an algebraic group over $\kk$. Here $\CI_{n}$ is the $n$th congruence subgroup of $\CI$. The conjugation action of $\Omega^{\ex}$ preserves $\CI_{n}$ and induces an action on $H_n$ for each $n$. We consider the group extension $H^{\ex}_n=H_n \rtimes \<\dot \t\>$. Then $\g \in \CI \dot \t$ corresponds to a family $\g_n \in H_n^{\ex}$. Let $\g_n=s_n u_n=u_n s_n$ be the Jordan decomposition of $\g_n \in H^{\ex}_n$. By the uniqueness of Jordan decomposition, both $\{s_n\}_n$ and $\{u_n\}_n$ form compatible families. Let $s=\lim_n s_n \in \CI \rtimes \<\dot \t\>$ and $u=\lim_n u_n \in \CI \rtimes \<\dot \t\>$. Then $\g=s u=u s$. By the proof of~\cite[Lemma B.1.7]{BV}, $s$ is strongly semisimple and $u$ is topologically unipotent. Then $s u$ is the topological Jordan decomposition of $\g$.

We also have the following simple criterion for topological unipotent elements. Let $U_{\CI}$ be the pro-unipotent radical of $\CI$ and $\bar \CI=\CI/U_{\CI}$ be the reductive quotient, which is a torus over $\kk$. Let $\g \in \CI \rtimes \Omega^{\ex}$ and $\bar \g$ be the image of $\g$ in $\bar \CI \rtimes \Omega^{\ex}$. It is easy to see that $\g$ is topologically unipotent if and only if $\bar \g$ is a unipotent element in the (disconnected) reductive group $\bar \CI \rtimes \Omega^{\ex}$. When the residue characteristic is large, the only unipotent elements in $\bar \CI \rtimes \Omega^{\ex}$ is the identity element. In this case, the topologically unipotent elements in $\CI \rtimes \Omega^{\ex}$ are the elements in $U_{\CI}$, and the topologically unipotent elements of $\breve G$ are all contained in $\breve G_{\af}$. However, this is not the case for small residue characteristic. 

\begin{example}
Let $\bG=\operatorname{PGL}_2$ and $\breve F=\bar \BF_2(\!(\e)\!)$. Then $\begin{bmatrix} \e & 1 \\ \e & \e+\e^2 \end{bmatrix}$ is a topologically unipotent regular semisimple element. However, this element is not contained in $\breve G_{\af}$. 
\end{example}

\subsection{Equi-dimensionality}
To prove Theorem~\ref{thm: aff}, one may reduce consideration to that of the adjoint group case. In particular, one may assume that $\g$ is a bounded element. 

Let $\g=s u$ be the topological Jordan decomposition. For any strongly semisimple element $s \in \breve G$, we denote by $\bG^{\mathrm{sc}}_s$ the simply connected covering of the derived group of the centralizer of $s$. By~\cite[Lemma B.2.3]{BV}, there is a universal homeomorphism between the affine Springer fiber $(\Fl^{\bG})^{\g}$ and a disjoint union of copies of the affine Springer fibers $(\Fl^{\bG^{\mathrm{sc}}_s})^{\g}$. 

It suffices to prove the following: 
\begin{enumerate}
\item[(a)] {\it Suppose that $\bG$ is a semisimple, simply connected group. Let $\g$ be a topological unipotent element in $\breve G^{\ex}$. Then the affine Springer fiber $\Fl^\g$ is connected and equi-dimensional.}
\end{enumerate}
When $\bG$ is split over $\BC(\!(\e)\!)$, the result is due to Kazhdan and Lusztig~\cite[\S 4, Proposition 1 and Lemma 2]{KL}. The general case follows from their proof, by paying extra attention to the Springer fibers for unipotent elements in disconnected reductive groups (only occurring in small characteristic) as in~\cite[Chap. II, \S 1]{Sp}. We give a sketch of the proof for the convenience of  readers. 

Suppose that $\g \in \breve G \dot \t$ for some $\t \in \Omega^{\ex}$. Let $\tilde \BS_{\t\text{-sph}}$ be the set of spherical $\Ad(\t)$-orbits on $\tilde \BS$. For any $K \in \tilde \BS_{\t\text{-sph}}$, we consider the projection map $\pi: \Fl \to \breve G/\CP_K$. By Steinberg's theorem, the map $\pi$ induces a surjection from $\Fl^\g$ to $Y_{K, \t}(\g)=(\breve G/\CP_K)^\g$, and the fiber over $g \CP_K$ is the Springer fiber in the finite flag $\CP_K/\CI$ associated with the element $\overline{g \i \g g}$, where $\overline{g \i \g g}$ is the image of $g \i \g g$ in the (disconnected) reductive group $\bar \CP_K \rtimes \Omega^{\ex}$. Since $\g$ is topologically unipotent, $\overline{g \i \g g}$ is unipotent. By~\cite[Chap. II, Lemme 1.3]{Sp}, any fiber of $\pi: \Fl^\g \to (\breve G/\CP_K)^\g$ is either a point or a one-dimensional curve. In the latter case, we call this one-dimensional curve a {\it line of type $K$}.

For connectedness, we follow the proof of~\cite[\S 4, Lemma 2]{KL}. Let $z, z' \in \Fl^\g$. Let $w$ be the relative position of $z$ and $z'$. Then $w$ is also the relative position of $\dot \t \cdot z$ and $\dot \t \cdot z'$. Thus $w$ commutes with $\t$. For any $K \in \tilde \BS_{\t\text{-sph}}$, we denote by $w_K$ the longest element in the finite Weyl group $W_K$. Since $\bG$ is semisimple and simply connected, the Iwahori--Weyl group $\tW$ is the affine Weyl group $W_{\af}$. By~\cite[Theorem A.8]{Lu-unequal}, the centralizer of $\t$ in $\tW$ is the affine Weyl group generated by $\{w_K\}_{K \in \tilde \BS_{\t\text{-sph}}}$. Therefore   we have a length-additive expression $w=w_{K_1} \cdots w_{K_n}$ for some $K_1, \ldots, K_n \in \tilde \BS_{\t\text{-sph}}$. Then there exists a unique sequence $z=z_0, z_1, \ldots, z_n=z'$ of elements in $\Fl^\g$ such that for any $0<i \le n$, $z_{i-1}$ and $z_{i}$ are in the same fiber of the projection map $\Fl^\g \to (\breve G/\CP_{K_{i}})^\g$. Hence  $z_{i-1}$ and $z_{i}$ are connected by a line of type $K_{i}$. In particular, $z$ and $z'$ are connected, and hence $\Fl^\g$ is connected. 

Finally, we prove  equi-dimensionality. Let $Y$ be an irreducible component of $\Fl^\g$. Let $Y_0$ be an irreducible component of $\Fl^\g$ of dimension $d$, where $d=\dim \Fl^\g$. Let $z \in Y_0$ and $z'$ be a generic point in $Y$. Let $z=z_0, z_1, \ldots, z_n=z'$ be a sequence of elements in $\Fl^\g$ such that for any $0<i \le n$, $z_{i-1}$ and $z_{i}$ are connected by a line $\CL_{i}$ of type $K_{i}$. By~\cite[\S 4, Lemma 1]{KL}, there exists an irreducible component $Y_i$ of $\Fl^\g$ such that $\dim Y_i=d$ and $\CL_i \subset Y_i$. In particular, $z' \in \CL_n \subset Y_n$. Since $z'$ is a generic point in $Y$, we must have $Y=Y_n$. In particular, $\dim Y=d$. 

\subsection{Nonemptiness of the regular locus}
When $\bG$ is split over $\BC(\!(\e)\!)$, and $K$ is hyperspecial, the result is due to Kazhdan and Lusztig~\cite[\S 4, Corollaries 1 and 2]{KL}. The general case follows from their proof, by paying extra attention to the Springer fibers for disconnected reductive groups. We give a sketch of the proof for the convenience of  readers. 

Let $d=\dim \Fl^\g$. Let $\rho: \tW^{\ex} \to \Aut(H_{2 d}(\Fl))$ be the natural action of $\tW^{\ex}$ on $H_{2 d}(\Fl)$. Suppose that $\g \in \breve G_{\af} \dot \t$ for some $\t \in \Omega^{\ex}$. Let $\hat V_\g$ be the image of $H_{2 d}(\Fl^\g)$ in $H_{2 d}(\Fl)$. Similar to~\cite[\S 4, Lemma 7]{KL}, $\hat V_\g$ is a $Z_{\tW^{\ex}}(\t)$-invariant subspace in $H_{2 d}(\Fl)$. Note that $Z_{W_K}(\t)$ is a finite subgroup of $Z_{\tW^{\ex}}(\t)$. Similar to~\cite[\S 4, Lemma 8]{KL}, $(\hat V_{\g})^{Z_{W_K}(\t)} \neq \{0\}$. For any irreducible component $Y$ of $\Fl^\g$, we denote by $[Y]$ the homology class represented by $Y$. Let $T=\sum_{w \in Z_{W_K}(\t)} \rho(w)$. Since $(\hat V_{\g})^{Z_{W_K}(\t)} \neq \{0\}$, $T[Y] \neq 0$ for some irreducible component $Y$ of $\Fl^\g$. In particular, $(\id+\rho(w_{K_1}))[Y] \neq 0$ for any $\Ad(\t)$-orbit $K_1$ in $K$. Similar to~\cite[\S 4, Lemma 9]{KL}, $Y$ is not a union of lines of type $K_1$. 

Let $y$ be a generic point in $Y$. Then $y$ is not in the line of type $K_1$ for any $\Ad(\t)$-orbit $K_1$ in $K$. Let $y'$ be the image of $y$ in $\breve G/\CP_K$. Then $y' \in (\breve G/\CP_K)^\g=Y_{K, \t}(\g)$. Suppose that $y'=g \CP_K$. The fiber of $y'$ for the projection map $\Fl^\g \to Y_{K, \t}(\g)$ is the Springer fiber of $\CP_K/\CI$ associated with the element $\overline{g \i \g g}$. This fiber does not contain any line of type $K_1$ through $y$. By~\cite[Chap. II, Corollaire 1.7]{Sp}, this fiber is $0$-dimensional. Hence $\overline{g \i \g g}$ is a regular element in $\overline{\CP_K} \rtimes \<\bar \t\>$. Thus $y' \in Y^{\reg}_{K, \t}(\g)$. 

Let $\pi: \Fl \to \breve G/\CP_K$ be the projection map. Then $\pi$ restricts to a surjective map $\Fl^\g \to Y_{K, \t}(\g)=(\breve G/\CP_K)^\g$. In particular, $\dim Y_{K, \t}(\g) \le \dim \Fl^\g$. On the other hand, since $Y^\reg_{K, \t}(g)$ is nonempty and open in $Y_{K, \t}(\g)$, $\pi \i(Y^\reg_{K, \t}(\g))$ is nonempty and open in $\Fl^\g$. By Theorem~\ref{thm: aff}~(1), $\dim \Fl^\g=\dim \pi \i(Y^\reg_{K, \t}(\g))$. Note that the fibers of the map $\pi: \pi \i(Y_{K, \t}(\g)) \to Y^{\reg}_{K, \t}(\g)$ are the Springer fibers in the finite flag associated with the regular elements and thus are $0$-dimensional. Hence $\dim(\pi \i(Y^\reg_{K, \t}(\g)))=\dim Y^{\reg}_{K, \t}(\g)$. So $\dim \Fl^\g=\dim Y^{\reg}_{K, \t}(\g) \le \dim Y_{K, \t}(\g)$. This finishes the proof. 

\section{Hodge--Newton decomposition}

\subsection{$P$-alcove elements}
We follow~\cite{GHKR2}. Let $\bP=\bM \bN$ be a semistandard parabolic subgroup. We say that $w \in \tW$ is a {\it $\bP$-alcove element} if $w$ is in the Iwahori--Weyl group $\tW_\bM$ of $\bM$ and for any relative root $\a$ of $\bG$ with the root subgroup $U_\a \subset \bN(\breve F)$, we have $U_\a \cap {}^{\dot w} \CI \subset U_\a \cap \CI$.

The Hodge--Newton decomposition for affine Deligne--Lusztig varieties was established in~\cite[Theorem 2.1.4]{GHKR2} for split groups and in~\cite[Theorem 3.3.1]{GHN} in general. We state the residually split group case here. In this case, the induced action of $\s$ on $\tW$ is trivial. 

\begin{theorem}\label{thm:HN}
Suppose that $\bP'=\bM' \bN'$ is a semistandard parabolic subgroup of $\bG'$ and $w \in \tW'$ is a $\bP'$-alcove element. Then $X^{\bG'}_w(b) \neq \emptyset$ implies that $[b] \cap \bM'(\breve F) \neq \emptyset$. Moreover, for $b \in \bM'(\breve F)$ with $\k_{\bM'}([b])=\k_{\bM'}([\dot w])$, the canonical immersion $\Fl_{\bM'} \to \Fl_{\bG'}$ induces $$\bJ_b^{\bM'}(F) \backslash X^{\bM'}_w(b) \cong \bJ_b^{\bG'} \backslash X^{\bG'}_w(b).$$ Here $\bJ_b^{\bM'}(F)$ is the $\s$-centralizer of $b$ in $\bM'(\breve F)$ and $X_w^{\bM'}(b)$ is the affine Deligne--Lusztig variety in the affine flag variety of $\bM'(\breve F)$. 
\end{theorem}

Note that $1 \in \tW$ is a $\mathbf B$-alcove element,  where $\mathbf B$ is the standard Borel subgroup of $\bG$. In this case, $Y_1(\g) \neq \emptyset$ for any $\g \in \CI$. However, there exists $\g \in \CI$ such that $\{\g\} \cap \bT(L)=\emptyset$. Thus the analog of Theorem~\ref{thm:HN} for affine Lusztig varieties does not hold in general.

The main purpose of this section is to show that the analog of Hodge--Newton decomposition for affine Lusztig varieties is valid when $w$ is a minimal length element in its conjugacy class. 

\subsection{The semistandard parabolic subgroup $\bP_{\nu_w}$}\label{sec:power}
For $w \in \tW$, let $\bP_{\nu_w}$ be the parabolic subgroup of $\bG$ generated by $\bT$ and the root subgroups $\mathbf U_\a$ with $\<\a, \nu_w\> \ge 0$, $\bM_{\nu_w}$ be the Levi subgroup of $\bP_{\nu_w}$ generated by $\bT$ and the root subgroups $\mathbf U_\a$ with $\<\a, \nu_w\>=0$, and $\bP^-_{\nu_w}$ be the opposite parabolic subgroup of $\bG$ generated by $\bT$ and the root subgroups $\mathbf U_\a$ with $\<\a, \nu_w\> \le 0$. We have the Levi decompositions $\bP_{\nu_w}=\bM_{\nu_w} \bN_{\nu_w}$ and $\bP^-_{\nu_w}=\bM_{\nu_w} \bN^-_{\nu_w}$, where $\bN_{\nu_w}$ is the unipotent radical of $\bP_{\nu_w}$ and $\bN^-_{\nu_w}$ is the unipotent radical of $\bP^-_{\nu_w}$.

Let $\breve M_{\nu_w}=\bM_{\nu_w}(\breve F)$ and  $\tW_{\nu_w}$ be the Iwahori--Weyl group of $\breve M_{\nu_w}$. Let $\CI_{\bM_{\nu_w}}=\CI \cap \breve M_{\nu_w}$. Then $\CI_{\bM_{\nu_w}}$ is an Iwahori subgroup of $\breve M_{\nu_w}$. We consider the length function on $\tW_{\nu_w}$ determined by $\CI_{\bM_{\nu_w}}$.  

By~\cite[Theorem 1.7]{Nie15}, if $w$ is a minimal length element in its conjugacy class in $\tW$, then $w$ is a $\bP_{\nu_w}$-alcove element. By~\cite[Proposition 4.5]{HN15}, $w$ is a minimal length element in its conjugacy class in $\tW_{\nu_w}$. Note that this is a nontrivial result, since the length function on $\tW_{\nu_w}$ is not the restriction of the length function on $\tW$. By definition, $\nu_w$ is central in the root system of $\bM_{\nu_w}$. By~\cite[Corollary 2.8]{HN14}, there exists a length-zero element $\t$ of $\tW_{\nu_w}$ and a $\dot \t$-stable standard parahoric subgroup $\CP$ of $\breve M_{\nu_w}$ such that $w=u \t$ for some element $u$ in the finite Weyl group $W_{\CP}$ of $\CP$. In particular, $\dot w \in \CP \dot \t$ and hence $\CI_{\bM_{\nu_w}} \dot w \CI_{\bM_{\nu_w}} \subset \CP \dot \t$. Note that for any $k \in \BN$, we have $(\CP \dot \t)^k=\CP \dot \t^k=\CP \dot w^k \CP$. Hence  we have the following:
\begin{enumerate}
\item[(a)] {\it Suppose that $w$ is a minimal length element in its conjugacy class in $\tW$. Then there exists a standard parahoric subgroup $\CP$ of $\breve M_{\nu_w}$ such that $$(\CI_{\bM_{\nu_w}} \dot w \CI_{\bM_{\nu_w}})^k \subset \CP \dot w^k \CP$$ for all $k \in \BN$.}
\end{enumerate}

\subsection{Method of successive approximations}
A key ingredient in the proof of Hodge--Newton decomposition in~\cite{GHKR2} is the method of successive approximations. 

We prove the following result for minimal length elements. The statement is similar to the $\s$-conjugation action related to $\bP$-alcove elements~\cite[Lemma 6.1.1]{GHKR2}, but the proof is simpler. In this subsection, we simply write ${}^g h$ for $g h g \i$. 

\begin{lemma}\label{lem:appox}
Suppose that $w$ is a minimal length element in its conjugacy class in $\tW$. For $n \in \BN$, let $\CI_n \subset \CI$ be the $n$th principal congruent subgroup of $\CI$. Set $N_n=\CI_n \cap \bN_{\nu_w}(\breve F), N^-_n=\CI_n \cap \bN_{\nu_w}^-(\breve F)$ and $M_n=\CI_n \cap \bM_{\nu_w}(\breve F)$. Then, for $m \in \CI_{\bM_{\nu_w}}$ and $n \in \BN$, we have the following: 
\begin{enumerate}
    \item Given $i_+ \in N_n$, there exists $b_+ \in N_n$ such that $b_+ i_+ {}^{m \dot w} b_+ \i \in N_{n+1}$.
    \item Given $i_- \in N^-_n$, there exists $b_- \in N^-_n$ such that ${}^{(\dot w m) \i} b_- i_- b_- \i \in N^-_{n+1}$.
\end{enumerate}
\end{lemma}

\begin{proof}
We prove part (1). Part (2) is proved in the same way. 

We simply write $\bM$ for $\bM_{\nu_w}$. Since $w$ is a $\bP_{\nu_w}$-alcove element, we have ${}^{\dot w} N_n \subset N_n$. Hence  ${}^{m \dot w} N_n \subset N_n$ and ${}^{(m \dot w)^k} N_n \subset N_n$ for all $k \in \BN$. By \S~\ref{sec:power}, there exists a standard parahoric subgroup $\CP$ of $\breve M_{\nu_w}$ such that $(m \dot w)^k \subset \CP \dot w^k \CP$ for all $k$. 

Note that $\CP \dot w^k \CP=\bigsqcup_{u \in W_{M}} \CI_{\bM} \dot u \dot w^k \CI_{\bM}$, where $W_M$ is the (finite) relative Weyl group of $\bM$. We choose $k \in \BN$ such that $\<k \nu_w, \a\> \ge \max\{\ell(u); u \in W_{M}\}+1$ for all roots in $\bN_{\nu_w}$. Let $\CI_{\bN}=\CI \cap \bN_{\nu_w}(\breve F)$ and $\CI_{\bN^-}=\CI \cap \bN_{\nu_w}^-(\breve F)$. We have the Iwahori decomposition $\CI=\CI_{\bN^-} \CI_{\bM} \CI_{\bN}$. Then, for any $u \in W_M$, $$\CI \dot u \dot w^k \CI=(\CI_{\bN} \CI_{\bM} \CI_{\bN^-}) \dot u \dot w^k (\CI_{\bN} \CI_{\bM} \CI_{\bN^-})=\CI_{\bN} \CI_{\bM} \dot u \dot w^k \CI_{\bM} \CI_{\bN^-}.$$ Thus $(m \dot w)^k \in N_0 \CI_{\bM} \dot u \dot w^k \CI_{\bM} N^-_0 \cap \breve M=\CI_{\bM} \dot u \dot w^k \CI_{\bM}$. Then $(m \dot w)^k=i \dot u \dot w^k i'$ for some $i, i' \in \CI_{\bM}$ and $u \in W_{M}$. So ${}^{(m \dot w)^k} N_n \subset {}^{i \dot w^k \dot u} N_n \subset {}^i N_{n+1} \subset N_{n+1}$. 

We take $b_+=({}^{(m \dot w)^{k-1}} i_+ \i) \cdots ({}^{m \dot w} i_+ \i) i_+ \i$. We then have $b_+ i_+ {}^{m \dot w} b_+ \i={}^{(m \dot w)^{k}} i_+ \in N_{n+1}$. 
\end{proof}

We now prove the following result, analogous to~\cite[Theorem 2.1.2]{GHKR2} for $\bP$-alcove elements and $\s$-conjugation action. 

\begin{theorem}\label{thm:conj}
Suppose that $w$ is a minimal length element in its conjugacy class in $\tW$. Let $\CI \times^{\CI_{\bM_{\nu_w}}} \CI_{\bM_{\nu_w}} \dot w \CI_{\bM_{\nu_w}}$ be the quotient of $\CI \times \CI_{\bM_{\nu_w}} \dot w \CI_{\bM_{\nu_w}}$ by the action of $\CI_{\bM_{\nu_w}}$ defined by $i \cdot (g, h)=(g i \i, i h i \i)$. Then 
the map $$\phi: \CI \times^{\CI_{\bM_{\nu_w}}} \CI_{\bM_{\nu_w}} \dot w \CI_{\bM_{\nu_w}} \to \CI \dot w \CI, \qquad (i, m) \mapsto i m i \i$$ is surjective. 
\end{theorem}

\begin{proof}
The proof is similar to that in~\cite[\S\S4 and 6]{GHKR2}. We simply write $\CI_{\bM}$ for $\CI_{\bM_{\nu_w}}$. Consider the following Cartesian diagram:
\[
\xymatrix{
({}^{\dot w} \CI \cap \CI) \times^{{}^{\dot w} \CI_{\bM} \cap \CI_{\bM}} \, \dot w \CI_{\bM} \ar[r]^-{\phi} \ar[d] & \dot w \CI \ar[d] \\
\CI \times^{\CI_{\bM}} \CI_{\bM} \dot w \CI_{\bM} \ar[r]^-{\phi} & \CI \dot w \CI 
}
\]

It suffices to prove that the map $({}^{\dot w} \CI \cap \CI) \times^{{}^{\dot w} \CI_{\bM} \cap \CI_{\bM}} \, \dot w \CI_{\bM} \to \dot w \CI$ is surjective. Following~\cite[\S 6]{GHKR2}, a generic Moy--Prasad filtration gives a filtration $\CI=\cup_{r \ge 0} \CI[r]$ with $\CI[r] \supset \CI[s]$ for $r<s$ such that 
\begin{enumerate}
    \item each $\CI[r]$ is normal in $\CI$;
    \item \begin{enumerate}
        \item either there exists a positive affine root $\a$ such that the embedding $U_\a \subset \CI$ induces an isomorphism $U_\a/U_{\a_+} \cong \CI[r]/\CI[r^+]$;
        \item or $(\bT(\breve F) \cap \CI[r])/(\bT(\breve F) \cap \CI[r^+]) \cong \CI[r]/\CI[r^+]$.
    \end{enumerate}
\end{enumerate}

We show that any element in $\dot w \CI_{\bM} \CI[r]$ is conjugate under ${}^{\dot w} \CI \cap \CI$ to an element of $\dot w \CI_{\bM} \CI[r^+]$ and that the conjugators can be taken to be small when $r$ is large. 

If $(\bT(\breve F) \cap \CI[r])/(\bT(\breve F) \cap \CI[r^+]) \cong \CI[r]/\CI[r^+]$, then $\CI_{\bM} \CI[r]=\CI_{\bM} \CI[r^+]$ and we may just take the conjugator $h_r$ to be $1$. If $U_\a/U_{\a_+} \cong \CI[r]/\CI[r^+]$, then for any $g \in \dot w \CI_{\bM} \CI[r]$, we may use Lemma~\ref{lem:appox} to produce a conjugator $h_r$ (suitably small when $r$ is large) such that $h_r g h_r \i \in \dot w \CI_{\bM} \CI[r^+]$. 

Now, for any $g \in \dot w \CI$, we construct the conjugators $h_r$ such that $$(h_r \cdots h_0) g (h_r \cdots h_0) \i \in \dot w \CI_{\bM} \CI[r^+].$$ Since $h_r$ is suitably small when $r$ is large, the convergent product $h=\lim_{r \to \infty} h_r \cdots h_0$ exists and $h g h \i \in \dot w \CI_{\bM}$. This finishes the proof. 
\end{proof}

\subsection{Hodge--Newton decomposition for ordinary conjugation}

We now prove the main result of this section. 

\begin{theorem}\label{thm:HN1}
Suppose that $w$ is a minimal length element in its conjugacy class in $\tW$. Then for any $\g \in \breve G$, $Y^{\bG}_w(\g) \neq \emptyset$ if and only if $\{\g\} \cap \CI_{\bM_{\nu_w}} \dot w \CI_{\bM_{\nu_w}} \neq \emptyset$. Moreover, for $\g \in \CI_{\bM_{\nu_w}} \dot w \CI_{\bM_{\nu_w}}$, we have $Z_{\breve M_{\nu_w}}(\g)=Z_{\breve G}(\g)$ and  $$Y^{\bM_{\nu_w}}_w(\g) \cong Y^{\bG}_w(\g).$$
\end{theorem}

\begin{proof}
Set $\bM=\bM_{\nu_w}$. We first prove that 
\begin{enumerate}
\item[(a)] If $\g_1, \g_2 \in \CI_{\bM} \dot w \CI_{\bM}$, and $g \in \breve G$ with $g \g_1 g \i=\g_2$, then $g \in \breve M$. 
\end{enumerate}

By definition, $g \g_1^k g \i=\g_2^k$ for any $k \in \BN$. Choose $k \in \BN$ such that $w^k=t^\l$ for some coweight $\l$. By \S\ref{sec:power}~(a), there exists a standard parahoric subgroup $\CP$ of $\breve M$ such that $\g_1^{l k}, \g_2^{l k} \in \CP t^{l \l}=\CP t^{l \l}$ for any $l \in \BN$. Then $\CP t^{l \l} g \cap g t^{l \l} \CP \neq \emptyset$. Set $\breve P=\bP_{\nu_w}(\breve F)$, $\breve N=\bN_{\nu_w}(\breve F)$, and $\breve N^-=\bN^-_{\nu_w}(\breve F)$. We may assume that $g \in \breve P \dot x \breve P$ for some $x \in W_0$. Note that $\breve P \dot x \breve P \cong (\breve N \cap \dot x \breve N^- \dot x \i) \times \breve M \dot x \breve P$. We write $g$ as $g_1 g_2$ with $g_1 \in (\breve N \cap \dot x \breve N^- \dot x \i)$ and $g_2 \in \breve M \dot x \breve P$. Then $g t^{l \l} \CP \subset g_1 \breve M \dot x \breve P$ and $\CP t^{l \l} g \subset \CP (t^{l \l} g_1 t^{l \l}) \breve M \dot x \breve P$. Hence  $g_1 \cap \CP (t^{l \l} g_1 t^{-l \l}) \breve M \neq \emptyset$ for all $l \in \BN$. Since $g_1 \in \breve N$, we have $\lim_{l \to \infty} t^{l \l} g_1 t^{-l \l}=1$. Thus $g_1 \cap \CP (t^{l \l} g_1 t^{-l \l}) \breve M \neq \emptyset$ for all $l \in \BN$ implies that $g_1=1$. Hence  $g \in \breve M \dot x \breve P$. By the same argument, we have $g \in \breve M \dot x \breve M$. We write $g$ as $g=g_3 \dot x g_4$ for $g_3, g_4 \in \breve M$. By definition, any element in $\breve M$ commutes with $t^{l \l}$. Then $\CP g_3 t^{l \l} \dot x g_4 \cap g_3 t^{l x(\l)} \dot x g_4 \CP \neq \emptyset$ and hence $t^{l (x(\l)-\l)} \in g_3 \i \CP g_3 \dot x g_4 \CP g_4 \i \dot x \i$ for all $l \in \BN$. Note that for fixed $g_3$ and $g_4$, $g_3 \i \CP g_3 \dot x g_4 \CP g_4 \i \dot x \i$ is bounded. Thus we have $x(\l)=\l$. Hence $x \in W_{\nu_w}$ and $g \in \breve M \dot x \breve M=\breve M$. 

(a) is proved.

We now prove the theorem. If $Y_w(\g) \neq \emptyset$, then $\{\g\} \cap \CI \dot w \CI \neq \emptyset$. By Theorem~\ref{thm:conj}, $\{\g\} \cap \CI_{\bM} \dot w \CI_{\bM} \neq \emptyset$. We assume that $\g \in \CI_{\bM} \dot w \CI_{\bM}$. Note that $\CI_{\bM} \dot w \CI_{\bM} \subset \CI \dot w \CI$. Let $\Fl^{\bM}$ be the affine flag variety of $\breve M$. The image of $Y^{\bM}_w(\g)$ under the immersion $\Fl^{\bM} \to \Fl$ is contained in $Y_w(\g)$. 

Suppose that $g \CI \in Y_w(\g)$. By Theorem~\ref{thm:conj}, up to multiplying a suitable element in $\CI$, we may assume that $g \i \g g \in \CI_{\bM} \dot w \CI_{\bM}$. By (a), $g \in \breve M$. This shows that the image of $Y^{\bM}_w(\g)$ equals $Y_w(\g)$. The proof is finished. 
\end{proof}

\subsection{Regular semisimple elements}
In \S \ref{sec:ass}, we defined a natural map $f_{\bG, \bG'}$ from the set of conjugacy classes of $\breve G$ to the set of $\s$-conjugacy classes of $\breve G'$. We have shown that the map $f_{\bG, \bG'}$ is surjective. Now we show that each fiber contains some regular semisimple conjugacy classes of $\breve G$. In other words, the restriction of the map $f_{\bG, \bG'}$ to the set of regular semisimple conjugacy classes of $\breve G$ is still surjective. This answers a question of the referee. 

Let $[b] \in B_\s(\bG')$. Recall that the induced action of $\s$ on $\tW'$ is trivial. Following \S \ref{two-diagram}, let $C=C_{[b]}$ be the straight conjugacy class of $\tW'$ associated with $[b]$. Then $[b]=\breve G' \cdot_\s \CI' \dot w \CI'$ for any minimal length element $w$ of $C$. We use the natural identification of $\tW$ and $\tW'$, and regard $w$ as a minimal length element of the conjugacy class $C$ in $\tW$. 

By Theorem~\ref{thm:HN1}, $\breve G \cdot \CI \dot w \CI=\breve G \cdot \CI_{\bM_{\nu_w}} \dot w \CI_{\bM_{\nu_w}}$. Recall that $\bT(L)_0$ is the unique parahoric subgroup of $\bT(L)$. It is easy to see that $\bT(L)_0 \subset \CI_{\bM_{\nu_w}}$ and $\bT(L)_0 \dot w$ contains regular semisimple elements of $\breve G$. Let $\g \in \bT(L)_0 \dot w \cap \breve G^{\rs}$. Then $C_{\{\g\}}=C_{[b]}$ and $f_{\bG, \bG'}(\{\g\})=[b]$.

\section{Lusztig varieties and Deligne--Lusztig varieties}

\subsection{The groups $\bH$ and $\bH'$} In this section, we compare classical Lusztig varieties and classical Deligne--Lusztig varieties. 

Let $\bH$ is a connected reductive algebraic group over an algebraically closed field $\kk$. We set $H=\bH(\kk)$. Let $\d$ be a group automorphism on $H$. Let $B$ be a $\d$-stable Borel subgroup of $H$. Let $W_H$ be the (finite) Weyl group of $H$ and $\BS_H$ be the set of simple reflections of $W_H$. The group automorphism $\d$ on $H$ induces a length-preserving group automorphism on $W_H$ and a bijection on $\BS_H$. We denote these maps still by $\d$. 

Let $\kk'$ be an algebraic closure of a finite field $\BF_q$ and $\s$ be the Frobenius map of $\kk'$ over $\BF_q$. Let $\bH'$ be a connected reductive group over $\BF_q$. We set $H'=\bH'(\kk')$. We still denote the Frobenius endomorphism on $H'$ by $\s$. Fix a $\s$-stable Borel subgroup $B'$ of $H'$. Let $W_{H'}$ be the (finite) Weyl group of $H'$ and $\BS_{H'}$ be the set of simple reflections of $W_{H'}$. We say that $\bH'$ is {\it associated} with the pair $(\bH, \d)$ if there exists a length-preserving isomorphism from $W_{H'}$ to $W_H$ such that the following diagram is commutative:
\[
\xymatrix{
W_{H'} \ar[r]^{\cong} \ar[d]_{\s} & W_H \ar[d]^{\d} \\
W_{H'} \ar[r]^{\cong}             & W_H
}
\]

\subsection{Lusztig varieties and Deligne--Lusztig varieties}

For any $\g \in H$ and $w \in W_H$, the {\it Lusztig variety} is defined by $$\CB_{w, \g, \d}=\{g B \in H/B; g \i \g \d(g) \in B \dot w B\}.$$ The isomorphism classes of  Lusztig varieties depend on the $\d$-twisted conjugacy class $\{\g\}$ of $\g$ in $H$ and $w \in W_H$. The variety $\CB_{w, \g, \d}$ was introduced by Lusztig~\cite{Lu79} in the case where $\g$ is a regular semisimple element of $H$ and was later extended to arbitrary elements in $H$ in~\cite{Lu-char}. In the case where $w=1$, the Lusztig variety $\CB_{1, \g, \d}=\{g B \in H/B; g \i \g \d(g) \in B\}$ is the Springer fiber of $\g$. We simply write $\CB_{\g, \d}$ for $\CB_{1, \g, \d}$.  Lusztig varieties play an important role in the theory of character sheaves. 

Similarly, for any $w' \in W_{H'}$, the {\it Deligne--Lusztig variety} $X_{w'}$ is defined to be $$X_{w'}=\{g B' \in H'/B'; g \i \s(g) \in B' \dot w B'\}.$$ By Lang's theorem, any element in $H'$ is $\s$-conjugate to $1 \in H'$. Thus, unlike  Lusztig varieties,  Deligne--Lusztig varieties  depend only on the element $w' \in W_{H'}$. 

It is easy to see that $X_{w'}$ is always nonempty and is of dimension $\ell(w')$.  Lusztig varieties, on the other hand, are more complicated, since there are two parameters $\{\g\}$ and $w$ involved. 

The main purpose of this section is to compare the dimension of  Lusztig varieties and  Deligne--Lusztig varieties for associated $(\bH, \d)$ and $\bH'$. In the rest of this section, we fix such $(\bH, \d)$ and $\bH'$ and identify $W_H$ with $W_{H'}$ and $\s=\d$ as group automorphisms on $W_H$.

The main result of this section is the following. 

\begin{proposition}\label{le}
	Let $\g \in H$ and $w \in W_H$. Then \[\dim \CB_{w, \g, \d} \le \dim X_w+\dim \CB_{\g, \d}.\]
\end{proposition}

The proof involves two parts: the base case for the minimal length elements $w$ in their $\d$-conjugacy classes and the reduction procedure to the base case. We discuss these in the rest of this section. 

\subsection{Base case}\label{sec:base} In this subsection, we prove Proposition~\ref{le} in the case where $w$ is of minimal length in its $\d$-twisted conjugacy class in $W_H$. Let $\supp(w)$ be the set of simple reflections that appear in some (or, equivalently, any) reduced expression of $w$ and $\supp_\d(w)=\bigcup_{i \in \BZ} \d^i(\supp(w))$. We say that $w$ is {\it $\d$-elliptic} if $\supp_\d(w)=\BS_H$. 

The following result was proved by Lusztig~\cite[Corollary 5.7]{Lu11b} in the case where $\d=id$. The same proof works for arbitrary $\d$. 
\begin{enumerate}
\item[(a)] {\it Let $w$ be a $\d$-elliptic element and of minimal length in its $\d$-twisted conjugacy class in $W_H$. If $\CB_{w, \g, \d} \neq \emptyset$, then it is of pure dimension $\ell(w)$.}
\end{enumerate}
Since $B \in \CB_{\g, \d}$, $\CB_{\g, \d}$ is always nonempty. Thus
\begin{enumerate}
\item[(b)] {\it For a $\d$-elliptic element $w$, we have $\dim \CB_{w, \g, \d} \le \dim \CB_{\g, \d}+\ell(w)$.}
\end{enumerate}

Next we consider the case where $w$ is not $\d$-elliptic. Let $J=\supp_\d(w)$ and $W_J$ be the parabolic subgroup of $W_H$ generated by $J$. Let $P_J \supset B$ be the standard parabolic subgroup of type $J$. Set $\CB^J_{\g, \d}=\{g P_J \in H/P_J; g \i \g \d(g) \in P_J\}$. Define the projection map $$\pi_J: H/B \to H/P_J, \qquad g B \mapsto g P_J.$$ Then, by definition, for any $x \in W_J$, $\pi_J$ sends $\CB_{x, \g, \d}$ into $\CB^J_{\g, \d}$.  

Let $g P_J \in \CB^J_{\g, \d}$. Set $\g'=g \i \g \d(g) \in P_J$. Let $L_J \subset P_J$ be the standard Levi subgroup, $B_J=B \cap L_J$ be the Borel subgroup of $L_J$, and $\g'_J \in L_J$ be the image of $\g'$ under the projection map $P_J \to L_J$. Since $J$ is $\d$-stable, both $P_J$ and $L_J$ are also $\d$-stable. For any $x \in W_J$, we have
\begin{align*}
\pi_J \i(g P_J) \cap \CB_{x, \g, \d} &=\{g g' B \in H/B; g' \in P_J, (g') \i g \i \g \d(g) \d(g') \in B \dot x B\} \\ & \cong \{g' B \in P_J/B; (g') \i \g' \d(g') \in B \dot x B\} \\ & \cong \{l B_J \in L_J/B_J; l \i \g'_J \d(l) \in B_J \dot x B_J\} \\ &=\CB^{L_J}_{x, \g'_J, \d}.
\end{align*}

Note that $w$ is of minimal length in its $\d$-twisted conjugacy class in $W$. Thus $w$ is of minimal length in its $\d$-twisted conjugacy class in $W_J$. By definition, $w$ is $\d$-elliptic in $W_J$. By (b),  $\dim \CB^{L_J}_{w, \g'_J, \d} \le \dim \CB^{L_J}_{\g'_J, \d}+\ell(w)$. The base case is established. 

\subsection{Reduction procedure} 
The reduction procedure is based on the following two ingredients: Deligne--Lusztig reduction and some combinatorial properties of the finite Weyl groups. 

The Deligne--Lusztig reduction method was introduced in~\cite[Proof of Theorem 1.6]{DL} for Deligne--Lusztig varieties. The same proof can be applied to Lusztig varieties. We state the result. 

\begin{proposition}\label{prop:DL}
Let $\g \in H$, $w \in W_H$ and $s \in \BS_H$.
\begin{enumerate}
\item If $\ell(s w \d(s))=\ell(w)$, then $\CB_{w, \g, \d}$ is isomorphic to $\CB_{s w \s(s), \g, \d}$ and $X_w$ is universally homeomorphic to $X_{s w \d(s)}$. 
\item If $s w \d(s)<w$, then we have the decomposition $$\CB_{w, \g, \d}=X_1 \sqcup X_2 \text{ and } X_w=X'_1 \sqcup X'_2,$$ where $X_1$ (resp. $X'_1$) is open in $\CB_{w, \g, \d}$ (resp. $X_w$) and is an $\BA^1$-bundle over $\CB_{s w, \g, \d}$ (resp. $X_{s w}$) with zero section removed, and $X_2$ (resp. $X'_2$) is an $\BA^1$-bundle over $\CB_{s w \d(s), \g, \d}$ (resp. $X_{s w \d(s)}$). 
\end{enumerate}
\end{proposition}

\begin{remark}
    It is worth pointing out that the morphism $X_w \to X_{s w \d(s)}$ in part (1) of Proposition \ref{prop:DL} involves the Frobenius morphism on $H/B$, while the morphism $\CB_{w, \g, \d} \to \CB_{s w \s(w), \g, \d}$ does not involve the Frobenius morphism. This is the reason that we only obtain the universal homeomorphisms among Deligne-Lusztig varieties, while we obtain the isomorphisms among Lusztig varieties. A similar phenomenon occurs for affine Deligne-Lusztig varieties and affine Lusztig varieties, as we will see in Proposition \ref{DLReduction1} and Proposition \ref{DLReduction}. 
\end{remark}

We now recall some combinatorial properties of Coxeter groups. Let $W$ be a Coxeter group, and let $\BS$ be the set of simple reflections of $W$. Following~\cite{GP93}, for $w, w' \in W$ and $s \in \BS$, we write $w \xrightarrow{s}_\d w'$ if $w'=s w \d(s)$ and $\ell(w') \le \ell(w)$. We write $w \to_\d w'$ if there is a sequence $w=w_0, w_1, \cdots, w_n=w'$ of elements in $W$ such that for any $k$, $w_{k-1} \xrightarrow{s}_\d w_k$ for some $s \in \BS$. We write $w \approx_\d w'$ if $w \to_\d w'$ and $w' \to_\d w$. It is easy to see that $w \approx_\d w'$ if $w \to_\d w'$ and $\ell(w)=\ell(w')$.

For any $\d$-conjugacy class $\CO$ in $W$, we denote by $\CO_{\min}$ the set of minimal length elements in $\CO$. The following result for finite Weyl groups is due to Geck and Pfeiffer~\cite{GP93}, Geck, Kim, and Pfeiffer~\cite{GKP}, and joint work of the present author and S.~Nie~\cite{HN12}. 

\begin{theorem}\label{ux}
	Let $W$ be a finite Weyl group and $\CO$ be a $\d$-conjugacy class of $W$ and $w \in \CO$. Then there exists $w' \in \CO_{\min}$ such that $w \to_\d w'$.
\end{theorem}

We now prove Proposition~\ref{le} by induction on $\ell(w)$. 
Suppose that $w$ is not of minimal length in its $\d$-twisted conjugacy class. By Theorem~\ref{ux}, there exists $w' \in W_H$ with $w \approx_\d w'$ and $s \in \BS_H$ such that $s w' \d(s)<w'$. By Proposition~\ref{prop:DL}, we have $$\dim \CB_{w, \g, \d}=\dim \CB_{w', \g, \d}=\max\{\CB_{s w', \g \d}, \CB_{s w' \d(s), \g, \d}\}+1.$$

Note that $\ell(s w'), \ell(s w' \d(s))<\ell(w)$. By the inductive hypothesis, $\dim \CB_{s w', \g, \d} \le \dim \CB_{\g, \d}+\ell(s w')$ and
$\dim \CB_{s w' \d(s), \g, \d} \le \dim \CB_{\g, \d}+\ell(s w' \d(s))$. So $$\dim \CB_{w, \g, \d} \le (\dim \CB_{\g, \d}+\ell(w)-1)+1=\dim \CB_{\g, \d}+\ell(w).$$ 

Hence the statement holds for $w$. The proof of Proposition~\ref{le} is finished. 

\subsection{Special case for regular elements}\label{sec:regg} 
There is one special case in which the dimensions of Lusztig varieties and Deligne--Lusztig varieties coincide. 

By definition, an element $\g \in H$ is called {\it $\d$-regular} if the $\d$-twisted conjugacy class of $\g$ is of maximal possible dimension. By a result of Springer~\cite[Proposition 5.1 and Lemma 5.3]{EG}\footnote{\cite{EG} only considers the case where $\d=\mathrm{id}$, but the same proof applies to arbitrary $\d$.}, $\dim \CB_{w, \g, \d}=\ell(w)$ for any $\d$-regular element $\g \in H$. In particular, $\dim \CB_{\g, \d}=0$ and hence 
\begin{enumerate}
\item[(a)] {\it For any $w \in W_H$ and $\d$-regular element $\g$ of $H$, we have $$\dim \CB_{w, \g, \d}=\dim X_w=\dim X_w+\dim \CB_{\g, \d}.$$}
\end{enumerate}

\section{Comparison of  dimensions}

\subsection{Deligne--Lusztig reduction}
We recall the ``reduction'' \`a la Deligne and Lusztig for the affine Deligne--Lusztig varieties. Recall that the group $\bG'$ is residually split and the induced action of the Frobenius endomorphism on the Iwahori--Weyl group is trivial. 

\begin{proposition}\label{DLReduction1}\cite[Proposition 3.3.1]{HZZ}
		Let $w\in \tW'$, $s\in \tilde \BS'$, and $b \in \breve G'$. If $\operatorname{char}(F') >0$, then the following two statements hold:
		\begin{enumerate}
			\item	If $\ell(s w s) =\ell(w)$, then there exists a $J_b$-equivariant morphism $X_w(b) \to X_{s w s}(b)$ that is a universal homeomorphism.
			\item
			If $\ell(s w s) =\ell(w)-2$, then $X_w(b)=X_1 \sqcup X_2$, where $X_1$ is a $J_b$-stable open subscheme $X$ of $X_w(b)$ and $X_2$ is its closed complement satisfying the following conditions:
			\begin{itemize}
				\item $X_1$ is $J_b$-equivariant universally homeomorphic to a Zariski-locally trivial $\bG_m$-bundle over $X_{s w}(b)$;
		        \item $X_2$ is $J_b$-equivariant universally homeomorphic to a Zariski-locally trivial $\BA^1$-bundle over $X_{s w s}(b)$.
			\end{itemize} 
		\end{enumerate}
		If $\operatorname{char}(F') =0$, then the above two statements still hold, but with $\BA^1$ and $\bG_m$ replaced by the perfections $\BA^{1, \mathrm{perf}}$ and $\bG^{\mathrm{perf}}_m$, respectively. 
\end{proposition}

Similarly, we have the following.

\begin{proposition}\label{DLReduction}
		Let $w\in \tW$, $s \in \tilde \BS$, and $\g \in \breve G^{\rs}$. Then the following two statements hold:
		\begin{enumerate}
			\item	If $\ell(s w s) =\ell(w)$, then there exists a $Z_{\breve G}(\g)$-equivariant isomorphism $Y_w(\g) \to Y_{s w s}(\g)$.
			\item
			If $\ell(s w s) =\ell(w)-2$, then $Y_w(\g)=Y_1 \sqcup Y_2$, where $Y_1$ is a $Z_{\breve G}(\g)$-stable open subscheme $X$ of $Y_w(\g)$ and $Y_2$ is its closed complement satisfying the following conditions:
			\begin{itemize}
				\item $Y_1$ is $Z_{\breve G}(\g)$-equivariant isomorphic to a Zariski-locally trivial $\bG_m$-bundle over $Y_{s w}(\g)$;
		        \item $Y_2$ is $Z_{\breve G}(\g)$-equivariant isomorphic to a Zariski-locally trivial $\BA^1$-bundle over $Y_{s w s}(\g)$.
			\end{itemize} 
		\end{enumerate}

  	If $\operatorname{char}(L) =0$, then the above two statements still hold, but with $\BA^1$ and $\bG_m$ replaced by the perfections $\BA^{1, \mathrm{perf}}$ and $\bG^{\mathrm{perf}}_m$, respectively. 
\end{proposition}

\subsection{Minimal length elements}
For any conjugacy class $C$ in $\tW$, we denote by $C_{\min}$ the set of minimal length elements in $C$. For $w, w' \in \tW$ and $s \in \tS$, we write $w \xrightarrow{s} w'$ if $w'=s w s$ and $\ell(w') \le \ell(w)$. We write $w \to w'$ if there is a sequence $w=w_0, w_1, \dots, w_n=w'$ of elements in $\tW$ such that for any $k$, $w_{k-1} \xrightarrow{s} w_k$ for some $s \in \tS$. We write $w \approx w'$ if $w \to w'$ and $w' \to w$. It is easy to see that $w \approx w'$ if $w \to w'$ and $\ell(w)=\ell(w')$. The following generalization of Theorem~\ref{ux} to the Iwahori--Weyl groups was established in~\cite[Theorem 2.10]{HN14}. 

\begin{theorem}\label{thm:min}
	Let $C$ be a conjugacy class of $\tW$ and $w \in C$. Then there exists $w' \in C_{\min}$ such that $w \to w'$.
\end{theorem}

We have the following results on nice representatives of minimal length elements, which will be used in \S\ref{sec:base1}.

\begin{theorem}\label{thm:ux}\cite[Propositions 2.4 and 2.7]{HN14}
Let $w \in \tW$. Then there exists a straight element $x \in \tW$, $K \subset \tilde \BS$ with $W_K$ finite, $x \in {}^K \tW^K$, and $\Ad(x)(K)=K$ and $u \in W_K$ such that $u x$ is a minimal length element in the conjugacy class of $w$ and $w \to u x$. 
\end{theorem}

\begin{remark}
    In \cite{HN14}, the results were stated in a more geometric way. We provide more details here. The element $w$ here corresponds to $\tilde w_A$ and the element $u x$ corresponds to $\tilde w_{A'}$ in \cite[Proposition 2.4]{HN14}. The decomposition $\tilde w_{A'}=u x$ follows from \cite[Proposition 2.7]{HN14}.
\end{remark}

We also point out that, in this case, $u$ is a minimal length element in the $\Ad(x)$-conjugacy class of $W_K$.

\begin{theorem}\label{thm:str-cyc}\cite[Theorem 3.8]{HN14}
Let $C$ be a straight conjugacy class of $\tW$. Then, for any straight elements $w$ and $w'$ in $C$, we have $w \approx w'$. 
\end{theorem}

\subsection{Affine Springer fiber associated with $\g$}\label{sec:affineS} 
Let $\g \in \breve G^{\rs}$ and $C$ be the straight conjugacy class of $\tW$ with $\g \in \{C\}$. Let $w, w'$ be minimal length elements of $C$. By Theorem~\ref{thm:str-cyc}, we have $w \approx w'$, and thus, by Proposition~\ref{DLReduction}, $Y_w(\g) \cong Y_{w'}(\g)$. By abusing notation, we set $Y_{\g}=Y_w(\g)$ for some minimal length element $w$ of $C$. The isomorphism class of $Y_{\g}$ is independent of the choice of such a minimal length representative. We call it the {\it affine Springer fiber associated with} $\{\g\}$. 

Let $w \in C$ be a minimal length element. By Theorem~\ref{thm:HN1}, $\{\g\} \cap \CI_{\bM_{\nu_w}} \dot w \CI_{\bM_{\nu_w}} \neq \emptyset$. We further assume that $\g \in \CI_{\bM_{\nu_w}} \dot w \CI_{\bM_{\nu_w}}$. By~\cite[Proposition 3.2]{HN14}, $w$ is a length-zero element in $\tW_{\nu_w}$. Hence $Y_\g \cong Y_w(\g) \cong Y^{\bM_{\nu_w}}_w(\g)$ is the affine Springer fiber of $\g$ in $\breve M_{\nu_w}$. In particular, $Y_\g$ is always non-empty. By Theorem~\ref{thm: aff}, we have 
\begin{enumerate}
\item[(a)] {\it Let $\g \in \breve G^{\rs}$. Then $Y_\g$ is equi-dimensional.}
\end{enumerate}

We now state the main result of this paper. 

\begin{theorem}\label{thm:le}
Suppose that $\bG'$ is associated with $\bG$. Let $\g$ be a regular semisimple element of $\breve G$ and $[b]=f_{\bG, \bG'}(\{\g\}) \in B_\s(\bG')$. Then, for any $w \in \tW$, we have 
\begin{enumerate}
    \item $Y_w(\g) \neq \emptyset$ if and only if $X_w(b) \neq \emptyset$;

    \item if $X_w(b) \neq \emptyset$, then $\dim Y_w(\g)=\dim X_w(b)+\dim Y_{\g}.$.
\end{enumerate}

In particular, $Y_w(\g)$ is either an empty set or finite-dimensional. 
\end{theorem}

The proof of Theorem~\ref{thm:le} involves two parts: the base case and the reduction procedure. The reduction procedure is similar to the proof of Proposition~\ref{le}, but the base case is more complicated. 

\subsection{Base case: upper bound}\label{sec:base1}
We consider the case where $w$ is of minimal length in its conjugacy class. Let $C$ be the straight conjugacy class of $\tW=\tW'$ associated with $\{\g\}$ and to $[b]$. 

Let $u$ and $x$ be the elements in Theorem~\ref{thm:ux} with $w \approx ux$. Let $\CP_K$ be the standard parahoric subgroup generated by $\CI$ and $\dot s$ for $s \in K$. By Proposition~\ref{DLReduction}~(1), $Y_w(\g) \cong Y_{u x}(\g)$. We have $\CI \dot u \dot x \CI \subset \CP_K \dot x \CP_{K}$. 

Let $\bar L_K=\CP_K/U_{\CP_K}$ be the reductive quotient of the parahoric subgroup $\CP_K$. Set $\d=\Ad(\dot x)$. Then $\d$ is a diagram automorphism on $\bar L_K$. Let $\bar B=\CI/U_{\CP_K}$ be a Borel subgroup of $\bar L_K$. Then $\bar B$ is stable under the action of $\d$. By Steinberg's theorem \cite{St}, any element in $\bar L_K$ is $\d$-conjugate to an element in $\bar B$. Thus, any element in $\CP_K \dot x \CP_K$ is conjugate by $\CP_K$ to an element in $\CI \dot x \CI \subset \{C'\}$, where $C'$ is the straight conjugacy class of $x$. In particular, $\CP_K \dot x \CP_K \subset \{C'\}$. 

If $C' \neq C$, then $\{C\} \cap \{C'\}=\emptyset$ and thus $\{\g\} \cap \CI \dot u \dot x \CI=\emptyset$. Hence $Y_{u x}(\g)=\emptyset$. By~\cite[Theorems 3.5 and 3.7]{He14}, $X_{u x}(b)=\emptyset$. The statement is obvious in this case. 

We now assume that $C'=C$. Let $\pi_K: \breve G/\CI \to \breve G/\CP_K$ be the projection map. Set $$Y=\{g \CP_K \in \breve G/\CP_K; g \i \g g \in \CP_K \dot x \CP_{K}\}.$$ Then, for any $u' \in W_K$, $\pi_K(Y_{u' x}(\g)) \subset Y$. 

Note that any element in $Y$ is of the form $g \CP_K$ for some $g \in \breve G$ with $g \i \g g=\g_1 \dot x$ for some $\g_1 \in \CP_K$. Let $\bar \g_1$ be the image of $\g_1$ in $\bar L_K$. Similar to the proof in \S\ref{sec:base}, $\pi_K \i(g \CP_K) \cap Y_{u' x}(\g)$ is isomorphic to the Lusztig variety $\{\bar g \bar B \in \bar L_K/\bar B; \bar g \bar \g_1 \d(\bar g) \i \in \bar B \dot u' \bar B\} \subset \bar L_K/\bar B$. Applying Proposition~\ref{le} to the fibers of the maps $Y_{u x}(\g) \to Y$ and $Y_{x}(\g) \to Y$, we obtain $$\dim Y_{ux}(\g)-\dim Y_{x}(\g) \le \ell(u).$$

By definition, $Y_\g=Y_x(\g)$. By~\cite[Theorem 4.8]{He14}, $\dim X_w(b)=\dim X_{u x}(b)=\ell(u)$. Thus, $\dim Y_w(\g)=\dim Y_{u x}(\g) \le \dim X_w(b)+\dim Y_\g$. 

\subsection{Base case: lower bound} We keep the notation as in \S\ref{sec:base1}. 

Let $\bM=\bM_{\nu_x}$ and $\CP=\breve M \cap \CP_K$. We have immersions $\Fl^{\bM} \to \Fl$ and $\breve M/\CP \to \breve G/\CP_K$. Set $$Y'=\{g \CP \in \breve M/\CP; g \i \g g \in \CP \dot x \CP\}.$$ The immersion $\breve M/\CP \to \breve G/\CP_K$ induces an immersion $Y' \to Y$. We have the following commutative diagram for any $u' \in W_K$:
\[
\xymatrix{
Y^{\bM}_{u' x}(\g) \ar[r] \ar[d] & Y_{u' x}(\g) \ar[d] \\
Y' \ar[r] & Y
}
\]

We have shown in \S\ref{sec:base1} that any element in $\CP_K \dot x \CP_K$ is conjugate by $\CP_K$ to an element in $\CI \dot x \CI$. Thus the map $Y_x(\g) \to Y$ is surjective. By Theorem~\ref{thm:HN1}, the map $Y^{\bM}_x(\g) \to Y_x(\g)$ is an isomorphism. The commutative diagram above for $u'=1$ implies that the immersion $Y' \to Y$ is an isomorphism. In particular, we may and will assume that $\g \in \CP \dot x \CP$. 

Set $\d=\Ad(\dot x)$. Then $\d$ is a group automorphism on $\breve M$ that stabilizes its Iwahori subgroup $\CI_{\bM}=\breve M \cap I$. Moreover, $\CP$ is a standard $\d$-stable parahoric subgroup of $\breve M$. Thus $Y'$ is the affine Springer fiber in the partial affine flag variety $\breve M/\CP$ associated with the element $\g \dot x \in \breve M^{\ex}$. Let $Y^{', \reg} \subset Y'$ be the regular locus defined in \S\ref{sec:reg}. By Theorems~\ref{thm: aff} and~\ref{thm:HN1}, $$\dim Y^{', \reg}=\dim Y'=\dim Y^{\bM}_x(\g)=\dim Y_x(\g)=\dim Y_\g.$$

Consider the following commutative diagram:
\[
\xymatrix{
(\pi'_K) \i (Y^{', \reg}) \cap Y^{\bM}_{u x}(\g) \ar@{^{(}->}[r] \ar[d]^{\pi'_K} & Y^{\bM}_{u x}(\g) \ar[r]^\cong \ar[d]^{\pi'_K} & Y_{u x}(\g) \ar[d]^{\pi_K} \\
Y^{', \reg} \ar@{^{(}->}[r] & Y' \ar[r]^\cong & Y
}
\]
where $\pi_K: \Fl \to \breve G/\CP_K$ and $\pi'_K: \Fl^{\bM} \to \breve M/\CP$ are the projection maps. 

For any $g \CP \in Y^{', \reg}$, the fiber \begin{align*} (\pi'_K) \i (g \CP) \cap Y^{\bM}_{u x}(\g) &=\{g g_1 \CI_{\bM} \in g \CP/\CI_{\bM}; g_1 \i (g \i \g g) g_1 \in \CI_{\bM} \dot u \dot x \CI_{\bM}\}.
\end{align*}

Let $\bar \CP$ be the reductive quotient of $\CP$ and $\bar \g_1$ be the image of $g \i \g g$ in $\bar \CP$. Since $g \CP \in Y^{', \reg}$, $\bar \g_1$ is a $\d$-regular element of $\bar \CP$. Similar to the proof in \S\ref{sec:base}, $(\pi'_K) \i (g \CP) \cap Y^{\bM}_{u x}(\g)$ is isomorphic to the Lusztig variety associated with the $\d$-regular element $\bar \g_1$ of $\bar \CP$. By \S\ref{sec:regg}, $\dim \bigl((\pi'_K) \i (g \CP) \cap Y^{\bM}_{u x}(\g)\bigr)=\ell(u)$ and hence $\dim\bigl((\pi'_K) \i (Y^{', \reg}) \cap Y^{\bM}_{u x}(\g)\bigr)=\dim Y^{', \reg}+\ell(u)$. 

By Theorem~\ref{thm:HN1}, $\dim Y_{u x}(\g)=\dim Y^{\bM}_{u x}(\g)$. We then have 
\begin{align*} \dim Y_{u x}(\g) &=\dim Y^{\bM}_{u x}(\g) \ge \dim\bigl((\pi'_K) \i (Y^{', \reg}) \cap Y^{\bM}_{u x}(\g)\bigr) \\ &=\dim Y^{', \reg}+\ell(u)=\dim X_{ux}(b)+\dim Y_\g.
\end{align*}

This finishes the proof for the base case.

\subsection{Reduction procedure}\label{red-ineq} 
Let $w \in \tW$. Suppose that $w \notin \tW_{\min}$ and that Theorem~\ref{thm:le} holds for any $w' \in \tW$ with $\ell(w')<\ell(w)$. By Theorem~\ref{ux}, there exists $w' \in \tW$ with $w \approx w'$ and $s \in \tilde \BS$ such that $s w' s<w'$.

By Propositions~\ref{DLReduction} and~\ref{DLReduction1}, we have 
\begin{gather*}
\dim X_w(b)=\dim X_{w'}(b)=\max\{\dim X_{s w'}(b), \dim X_{s w' s}(b)\}+1, \\
\dim Y_w(\g)=\dim Y_{w'}(\g)=\max\{\dim Y_{s w'}(\g), \dim Y_{s w' s}(\g)\}+1. 
\end{gather*} 
Here we use the convention that $\dim \emptyset=-\infty$ and $-\infty+1=-\infty$. Note that $\ell(s w'), \ell(s w' s)<\ell(w)$. By the inductive hypothesis, we have 
\begin{enumerate}
	\item $Y_{s w'}(\g) \neq \emptyset$ if and only if $X_{s w'}(b) \neq \emptyset$. In this case, $$\dim Y_{s w'}(\g)-\dim X_{s w'}(b)=\dim Y_\g;$$ 
	\item $Y_{s w' s}(\g) \neq \emptyset$ if and only if $X_{s w' \s}(b) \neq \emptyset$. In this case, $$\dim Y_{s w' s}(\g)-\dim X_{s w' s}(b)=\dim Y_\g.$$ 	
\end{enumerate}
Thus $Y_w(\g) \neq \emptyset$ if and only if $X_w(b) \neq \emptyset$. In this case, \begin{align*} & \dim Y_w(\g)-\dim X_{w}(b) \\ & =\max\{\dim Y_{s w'}(\g), \dim Y_{s w' s}(\g)\}-\max\{\dim X_{s w'}(b), \dim X_{s w' s}(b)\} \\ &=\dim Y_\g.\end{align*}

Hence Theorem~\ref{thm:le} holds for $w \in \tW$. 

\subsection{Explicit descriptions}
Let $\le$ be the dominance order on the set of dominant rational coweights, that is, $\l_1 \le \l_2$ if and only if $\l_2-\l_1$ is a non-negative rational linear combination of simple coroots. 

We recall the definitions of the defect and the virtual dimension. 

For $b \in \breve G'$, the {\it defect} of $b$ is defined to be $\operatorname{def}(b)=\rank_F \bG'-\rank_F \bJ_b$, where for a reductive group $\bH$ defined over $F$, $\rank_F$ is the $F$-rank of the group $\bH$. The defect may also be defined in the following way, depending only on the Iwahori--Weyl group, not the reductive group. Let $V_{w}=\{v \in V; w \cdot v=v+\nu_{w}\}$. For any straight conjugacy class $C$ of $\tW$, we define the {\it defect} of $C$ by $$\operatorname{def}(C)=\dim V-\dim V_{w} \text{ for any } w \in C.$$ By~\cite[\S 1.9]{Ko06}, $\operatorname{def}(C)=\operatorname{def}(b)$, where $[b]=[\dot w]$ for any $w \in C$.

Note that any element in $\tW$ can be written in a unique way as $x t^\l y$, where $x, y \in W_0$, $\l$ is a dominant coweight such that $t^\l y \in {}^{\BS_0} \tW$. We set $\eta(x t^\l y)=y x \in W_0$. For any $w \in \tW$ and $b \in \breve G'$, the {\it virtual dimension} is defined to be $$d_w(b)=\frac 12 \big( \ell(w) + \ell(\eta(w)) -\operatorname{def}(b)  \big)-\<\nu_{[b]}, \rho\>.$$ By~\cite[Theorem 2.30]{He-CDM}, we have $\dim X_w(b) \le d_w(b)$. 

We may reformulate the virtual dimension using the straight conjugacy class $C$ instead. Set $$d_w(C)=\frac 12 \big( \ell(w) + \ell(\eta(w))-\operatorname{def}(C)-\ell(C)\big).$$ Here $\ell(C)$ is defined as the length of any minimal length element in $C$. We have $\ell(C)=\<\nu_{[b]}, 2 \rho\>$ if $[b]=[\dot w]$ for any $w \in C$. Here, $\rho$ is the half sum of positive roots. 

We now state the result for the nonemptiness pattern and dimension formula for affine Deligne--Lusztig varieties in the affine Grassmannian. 

\begin{theorem}\label{thm:gr}
Let $\mu$ be a dominant coweight. Then $X_\mu(b) \neq \emptyset$ if and only if $\k(\mu)=\k(b)$ and $\nu_{[b]} \le \mu$. In this case, $$\dim X_\mu(b)=\<\mu-\nu_{[b]}, \rho\>-\frac{1}{2}\operatorname{def}(b).$$ 
\end{theorem}

The ``only if'' part of the nonemptiness pattern is due to Rapoport and Richartz~\cite{RR} and the ``if'' part is due to Gashi~\cite{Ga}. The dimension formula is proved in~\cite{GHKR06} and~\cite{Vi06} for split groups, in~\cite{Ham15} and~\cite{Zhu} for unramified groups, and in~\cite[Theorem 2.29]{He-CDM} in general. 

Combining Theorem~\ref{thm:gr} with Theorem~\ref{thm:le}, we obtain the following description of the nonemptiness pattern and dimension formula for affine Lusztig varieties in the affine Grassmannian. 

\begin{theorem}\label{thm:gasf-gr}
Let $\mu$ be a dominant coweight and $\g \in \breve G^{\rs}$. Then $Y_\mu(\g) \neq \emptyset$ if and only if $\k(\mu)=\k(\g)$ and $\nu_{\{\g\}} \le \mu$. In this case, $$\dim Y_\mu(\g)=\<\mu-\nu_{\{\g\}}, \rho\>-\frac{1}{2}\operatorname{def}(C_{\{\g\}})+\dim Y_{\g}.$$
\end{theorem}

\begin{proof}
    Note that $\CK^{\mathrm{sp}} t^\mu \CK^{\mathrm{sp}}=\bigsqcup_{w \in W_0 t^\mu W_0} \CI \dot w \CI$. Set $Y_1=\sqcup_{w \in W_0 t^\mu W_0} Y_w(\g)$. We have the following Cartesian diagram
    \[
\xymatrix{Y_1 \ar[r] \ar[d] & \Fl \ar[d] \\
Y_\mu(\g) \ar[r] & \Gr.
}
    \]
    Note that the fibers of the projection map $\Fl \to \Gr$ are isomorphic to $\CK^{\mathrm{sp}}/\CI$. Then we have $$\dim Y_\mu(\g)=\dim Y_1-\dim \CK^{\mathrm{sp}}/\CI=\dim Y_1-\ell(w_0).$$

    Similarly, we set $X_1=\sqcup_{w \in W_0 t^\mu W_0} X_w(b)$. We have $\dim X_\mu(b)=\dim X_1-\ell(w_0)$. 

    In particular, $X_1 \neq \emptyset$ if and only if $X_\mu(b) \neq \emptyset$ and $Y_1 \neq \emptyset$ if and only if $Y_\mu(\g) \neq \emptyset$. By Theorem \ref{thm:gasf-gr}, $X_w(b) \neq \emptyset$ for some $w \in W_0 t^\mu W_0$ if and only if $\k(\mu)=\k(b)$ and $\nu_{[b]} \le \mu$. By \S\ref{sec:ass}, $\k(\g)=\k(b)$ and $\nu_{\{\g\}}=\nu_{[b]}$. Hence, by Theorem \ref{thm:le}, $Y_\mu(\g) \neq \emptyset$ if and only $\k(\mu)=\k(\g)$ and $\nu_{\{\g\}} \le \mu$.

    Suppose that $\k(\mu)=\k(\g)$ and $\nu_{\{\g\}} \le \mu$. By Theorem~\ref{thm:le} and Theorem~\ref{thm:gasf-gr} we have 
    \begin{align*}
        \dim Y_\mu(\g) &=\dim Y_1-\ell(w_0) \\ &=\max_{w \in W_0 t^\mu W_0} \dim Y_w(\g)-\ell(w_0) \\ &=\max_{w \in W_0 t^\mu W_0} \dim X_w(b)+\dim Y_\g-\ell(w_0) \\ &=\dim X_1-\ell(w_0)+\dim Y_\g \\ &=\dim X_w(b)+\dim Y_\g \\ &=\<\mu-\nu_{\{\g\}}, \rho\>-\frac{1}{2}\operatorname{def}(C_{\{\g\}})+\dim Y_{\g}. \qedhere
    \end{align*}
\end{proof}

We do not have a complete description of the nonemptiness pattern and dimension formula for affine Deligne--Lusztig varieties in the affine flag variety. However,~\cite{He20} provides the explicit answer for almost all cases. Combining Theorem~\ref{thm:le} with~\cite[Theorem 6.1]{He20}, we obtain the following result on affine Lusztig varieties in the affine flag variety. 

\begin{theorem}\label{thm:gasf-fl}
Let $\g \in \breve G^{\rs}$. Suppose that $\<\mu, \a\> \ge 2$ for any simple root $\a$ and $\nu_{\{\g\}}+2 \rho^\vee \le \mu$. Then, for any $x, y \in W_0$, $Y_{x t^\mu y}(\g) \neq \emptyset$ if and only if $\k(\g)=\k(\mu)$ and $\supp(y x)=\BS$. In this case, $$\dim Y_{x t^\mu y}(\g)=d_{x t^\mu y}(C_{\{\g\}})+\dim Y_{\g}.$$
\end{theorem}

\end{document}